\documentclass[a4paper,reqno,11pt]{amsart}

\usepackage{amsmath, amsfonts, amssymb, amsthm, amscd}
\usepackage{graphicx, epsfig}
\usepackage{psfrag}
\usepackage{perpage}
\usepackage{url}
\usepackage{color}
\usepackage[utf8]{inputenc}
\usepackage[T1]{fontenc}
\usepackage{microtype}

\usepackage[a4paper,scale={0.72,0.74},marginratio={1:1},footskip=7mm,headsep=10mm]{geometry}

\usepackage{hyperref}

\numberwithin{equation}{section}

\newtheorem{theorem}{Theorem}[section]
\newtheorem{lemma}[theorem]{Lemma}
\newtheorem{proposition}[theorem]{Proposition}

\newtheorem{remark}[theorem]{Remark}

\newcommand{\bbE}{{\ensuremath{\mathbb E}} }

\newcommand{\bbP}{{\ensuremath{\mathbb P}} }

\newcommand{\cD}{{\ensuremath{\mathcal D}} }

\newcommand{\cH}{{\ensuremath{\mathcal H}} }

\newcommand{\cL}{{\ensuremath{\mathcal L}} }

\newcommand{\cN}{{\ensuremath{\mathcal N}} }

\newcommand{\cP}{{\ensuremath{\mathcal P}} }

\newcommand{\cR}{{\ensuremath{\mathcal R}} }

\DeclareMathSymbol{\leqslant}{\mathalpha}{AMSa}{"36} % nicer `smaller or equal'
\DeclareMathSymbol{\geqslant}{\mathalpha}{AMSa}{"3E} % nicer `larger or equal'
\DeclareMathSymbol{\eset}{\mathalpha}{AMSb}{"3F}     % nicer `emptyset'
\newcommand{\dd}{\text{\rm d}}             % a straight d for differentials
\newcommand{\R}{\mathbb{R}}

\newcommand{\Z}{\mathbb{Z}}
\newcommand{\N}{\mathbb{N}}

\newcommand{\PEfont}{\mathrm}

\DeclareMathOperator{\p}{\ensuremath{\PEfont P}}
\DeclareMathOperator{\e}{\ensuremath{\PEfont E}}

\DeclareMathOperator{\bbvar}{\ensuremath{\mathbb{V}ar}}

\newcommand{\ind}{{\sf 1}}

\renewcommand{\epsilon}{\varepsilon} 
\renewcommand{\theta}{\vartheta} 
\renewcommand{\rho}{\varrho} 
\renewcommand{\phi}{\varphi}

\def\rc{\mathrm{c}}
\def\dd{\mathrm{d}}
\def\ee{\mathrm{e}}

\def\pin{\mathrm{pin}}
\def\cop{\mathrm{cop}}
\def\per{\mathrm{per}}
\def\inv{\mathrm{inv}}

\def\ann{\mathrm{ann}}
\def\que{\mathrm{que}}

\def\fin{\mathrm{fin}}

\def\f{\textsc{f}}
\def\M{\mathrm{M}}

\newenvironment{myenumerate}{%
\renewcommand{\theenumi}{\arabic{enumi}}%
\renewcommand{\labelenumi}{{\rm(\theenumi)}}%
\begin{list}{\labelenumi}
	{%
	\setlength{\itemsep}{0.4em}%
	\setlength{\topsep}{0.5em}%
	\setlength\leftmargin{2.45em}%
	\setlength\labelwidth{2.05em}%
	\setlength{\labelsep}{0.4em}%
	\usecounter{enumi}%
	}%
	}%
{\end{list}
}

{\end{myenumerate}} 

\newenvironment{myitemize}{%
\begin{list}{$\bullet$}% 
 	{%
	\setlength{\itemsep}{0.4em}%
	\setlength{\topsep}{0.5em}%
	\setlength\leftmargin{2.45em}%
	\setlength\labelwidth{2.05em}%
	\setlength{\labelsep}{0.4em}%
	}%
	}%
{\end{list}}

\renewenvironment{itemize}{
\begin{myitemize}}%
{\end{myitemize}} 

\MakePerPage[2]{footnote} % restarts footnote counter at every new page

\begin{document}

\title{Phase transitions for spatially extended pinning}

\author{Francesco Caravenna}
\address{Dipartimento di Matematica e Applicazioni\\
Universit\`a degli Studi di Milano-Bicocca\\
via Cozzi 53, 20125 Milano, Italy}
\email{francesco.caravenna@unimib.it}

\author{Frank den Hollander}
\address{Mathematisch instituut\\
Universiteit Leiden\\
Postbus 9512\\
2300 RA Leiden\\
The Netherlands}
\email{denholla@math.leidenuniv.nl}

\begin{abstract}
We consider a directed polymer of length $N$ interacting with a linear interface. The monomers 
carry i.i.d.\ random charges $(\omega_i)_{i=1}^N$ taking values in $\R$ with mean zero and 
variance one. Each monomer $i$ contributes an energy $(\beta\omega_i-h)\phi(S_i)$ to the 
interaction Hamiltonian, where $S_i \in \Z$ is the height of monomer $i$ with respect to the 
interface, $\phi\colon\,\Z\to [0,\infty)$ is the interaction potential, $\beta \in [0,\infty)$ is the 
inverse temperature, and $h \in \R$ is the charge bias parameter. The configurations 
of the polymer are weighted according to the Gibbs measure associated with the interaction 
Hamiltonian, where the reference measure is given by a Markov chain on $\Z$. We study 
both the \emph{quenched} and the \emph{annealed} free energy per monomer in the limit 
as $N\to\infty$. We show that each exhibits a phase transition along a critical curve in the 
$(\beta,h)$-plane, separating a \emph{localized phase} (where the polymer stays close to 
the interface) from a \emph{delocalized phase} (where the polymer wanders away from the 
interface). 
We derive variational formulas for the critical curves and we obtain upper and lower bounds on the 
quenched critical curve in terms of the annealed critical curve. In addition, for the special 
case where the reference measure is given by a \emph{Bessel random walk}, we derive 
the scaling limit of the annealed free energy as $\beta,h \downarrow 0$ in \emph{three different 
regimes} for the tail exponent of $\phi$. 

\medskip\noindent
$\ast$ Dedicated to the memory of Harry Kesten, who for over half a century generously 
showed the way in probability theory. 
\end{abstract}

\thanks{The research in this paper was supported by ERC Advanced Grant 267356 VARIS
and NWO Gravitation Grant 024.002.003 NETWORKS}

\date{\today}

\maketitle

%%%%%%%%%% SECTION 1 %%%%%%%%%%%%%%%%

\section{Introduction}
\label{s:intro}

\subsection{Motivation}
\label{ss:motivation}

\emph{Homogeneous pinning} models, where a directed polymer receives a reward for
every monomer that hits an interface, have been the object of intense study. Both discrete 
and continuous models have been analysed in detail, and a full understanding is available 
of the free energy, the phase diagram and the typical polymer configurations as a function 
of the underlying model parameters. \emph{Disordered pinning} models, where the reward 
depends on random weights attached to the interface or where the shape of the interface 
is random itself, are much harder to analyse. Still, a lot of progress has been made in past 
years, in particular, the effect of the disorder on the scaling properties of the polymer has 
been elucidated to considerable depth. For an overview the reader is referred to the 
monographs by Giacomin~\cite{cf:Gi1}, \cite{cf:Gi2} and den Hollander~\cite{cf:dHo}, 
the review paper by Caravenna, Giacomin and Toninelli ~\cite{cf:CaGiTo}, and references 
therein.  

\emph{Spatially extended pinning}, where the interaction of the monomers depends 
on their distance to the interface, remains largely unexplored. For a continuum model
with an interaction potential that decays polynomially with the distance, pinning-like results have 
been obtained in Lacoin~\cite{cf:La}. A continuum model for which the interaction potential is non-zero 
only in a finite window around the interface was analysed in Cranston, Koralov, 
Molchanov and Vainberg~\cite{cf:CrKoMoVa}. The goal of the present paper is to 
investigate what happens for more general interaction potentials, both for discrete and 
for continuous models with disorder.   

The remainder of this section is organised as follows. In Section~\ref{ss:model} we define 
our model, which consists of a directed polymer carrying random charges that interact 
with a linear interface at a strength that depends on their distance. In Sections~\ref{ss:quefe} 
and \ref{ss:annfe} we look at the quenched, respectively, the annealed free energy, and
discuss the qualitative properties of the phase diagram. In Section~\ref{ss:Besselscal}
we recall certain scaling properties of the Bessel random walk and its relation to the 
Bessel process, both of which play an important role in our analysis. In Section~\ref{ss:thmsgen} 
we state three theorems when the underlying reference measure (describing the polymer 
without interaction) is a Markov chain. In Section~\ref{ss:thmsweak} we state three theorems 
for the limit of weak interaction when the reference measure is the Bessel random walk, and 
show that this limit is related to the continuum version of our model when the reference 
measure is the Bessel process. In Section~\ref{ss:disc} we place the theorems in their 
proper perspective. In Section~\ref{ss:open} we list some open problems and explain 
how the proofs of the theorems are organised.  

%%%%%%%%%%%%%%

\subsection{The model}
\label{ss:model}

Let $\N_0=\N\cup\{0\}$. Our model has three ingredients:

\begin{itemize}
\item[(1)] 
An irreducible nearest-neighbour recurrent \emph{Markov chain} $S:=(S_n)_{n\in\N_0}$ 
on $\Z$ starting at $S_0=0$, with law $\p=\p_0$.
\item[(2)] 
An i.i.d.\ sequences of \emph{random charges} $\omega:=(\omega_n)_{n\in\N}$ on 
$\R$, with law $\bbP$.
\item[(3)] 
A non-negative function $\phi\colon \Z \to [0,\infty)$, playing the role of an \emph{interaction potential}, such that 
\begin{equation}
\label{eq:phidecay}
0 < \|\phi\|_\infty < \infty, \quad \lim_{x\to\infty} \phi(x) = 0, \quad \lim_{x\to-\infty} \phi(x) \text{ exists}.
\end{equation} 
\end{itemize}

\medskip\noindent
Our model is defined through the quenched partition function
\begin{equation}
\label{eq:model0}
Z_{N,\beta}^{\omega} := \e\left[ \ee^{\beta \sum_{n=1}^N \omega_n \phi(S_n)} \right],
\qquad N \in \N_0,
\end{equation}
which describes a directed polymer chain $n \mapsto (n,S_n)$ of length $N$ carrying 
charges $n \mapsto \omega_n$ that interact with a linear interface according to the 
interaction potential $x \mapsto \phi(x)$ at inverse temperature $\beta \in [0,\infty)$. 
Without loss of generality we may replace $\beta\omega_n$ by $\beta\omega_n - h$, 
with $h \in \R$ the charge bias parameter, and assume that $\omega$ is standardized, 
i.e.,
\begin{equation}
\label{eq:omass}
\bbE[\omega_n] = 0, \qquad \bbvar[\omega_n] = 1,
\end{equation}
after which \eqref{eq:model0} becomes
\begin{equation} 
\label{eq:model1}
Z_{N,\beta,h}^\omega := \e\left[ \ee^{\sum_{n=1}^N (\beta\omega_n - h) \phi(S_n)} \right],
\qquad N\in\N_0.
\end{equation}

Throughout the sequel we assume that
\begin{equation}
\label{eq:expmom}
\M(t) := \bbE\big[ \ee^{t\omega_1} \big] < \infty \qquad \forall\, t \in \R.
\end{equation}
Moreover, defining $\tau_1 := \inf\{n\in\N\colon\,S_n = 0\}$ to be the first return time of $S$ to 0, we 
assume that there exists an $\alpha \in [0,\infty)$ such that 
\begin{equation} 
\label{eq:astau}
\sum_{n\in\N} \p(\tau_1 = n) = 1, \qquad \p(\tau_1 = n)  = n^{-(1+\alpha)+o(1)}, 
\quad n \to \infty.
\end{equation}
Note that $\e(\tau_1)=\infty$ for all $\alpha \in (0,1)$.
If $S$ has period $2$, then the last asymptotics is assumed to run along $2\N$.

\begin{remark}{\bf [Bessel random walk]}
\label{rem:Bessel}
\rm An example of a Markov chain $S$ satisfying \eqref{eq:astau} that will receive special 
attention in this paper is the one with transition probabilities 
\begin{equation}
\label{eq:Besseltr}
\p(S_{n+1} = x \pm 1 \mid S_n = x) =: \tfrac12\big[1 \pm d(x)\big], \quad x \in \Z,
\end{equation}
where
\begin{equation}
\label{eq:dasymp}
d(x)= -d(-x), \quad x \in \Z, \qquad d(x) = -(\alpha-\tfrac12)\,x^{-1} + O(|x|^{-(1+\epsilon)}), 
\quad |x| \to \infty, 
\end{equation} 
for some $\alpha \in (0,1)$ and $\epsilon>0$. This choice, which is referred to as the 
\emph{Bessel random walk}, has a drift away from the origin ($\alpha<\tfrac12$) or 
towards the origin ($\alpha>\tfrac12$) that decays inversely proportional to the distance. The 
case $d(x) \equiv 0$ ($\alpha=\tfrac12$) corresponds to simple random walk. The Bessel random 
walk was studied by Lamperti~\cite{cf:Lam} and, more recently, by Alexander~\cite{cf:Al} 
(who actually considered the one-sided version $(|S_n|)_{n\in\N_0}$). It is known that \eqref{eq:astau} 
holds in a sharp form \cite[Theorem~2.1]{cf:Al}, namely, 
\begin{equation} \label{eq:astau1}
\p(\tau_1 = n) \sim c \, n^{-(1+\alpha)}, \qquad n\to\infty,
\end{equation} 
along $2\N$ for some $c \in (0,\infty)$. More refined asymptotics are available as well 
(see Section~\ref{ss:Besselscal} below).
\end{remark}

\begin{remark} \rm
The model defined in \eqref{eq:model1} provides a natural interpolation between the 
\emph{pinning model} and the \emph{copolymer model}, which correspond to the choices
\begin{equation}
\label{eq:phipincop}
\phi^\pin(x) = \ind_{\{x=0\}}, \qquad \phi^\cop(x) = \ind_{\{x \leq 0\}}.
\end{equation}
See Giacomin~\cite{cf:Gi1}, \cite{cf:Gi2} and den Hollander~\cite{cf:dHo} for details. 
Actually, in the copolymer model the interaction is via the bonds rather than the sites 
of the path, i.e., $\phi^\cop((x,y)) = \ind_{\{x+y \leq 0\}}$, but we will ignore such 
refinements. Moreover, the standard parametrisation of the disorder in the copolymer 
model is $-2\beta(\omega_n + h)$ rather than $\beta\omega_n-h$. Again, this is the 
same after a change of parameters. Our choice has the advantage that the free energy 
is jointly convex in $(\beta, h)$ and that the critical curve is non-negative (see 
Fig.~\ref{fig-critcurveque}).
\end{remark}

%%%%%%%%%%%%%%

\subsection{The quenched free energy}
\label{ss:quefe}

The \emph{quenched free energy} is defined by
\begin{equation}
\label{eq:f}
\f^\que(\beta,h) := \lim_{N\to\infty} \frac{1}{N} \log Z_{N,\beta,h}^\omega 
\quad \text{$\bbP$-a.s.\ and in $L^1(\bbP)$}.
\end{equation}
For the \emph{constrained} partition function
\begin{equation} 
\label{eq:Zc}
Z_{N,\beta,h}^{\omega,\rc} := \e\left[ \ee^{\sum_{n=1}^N (\beta\omega_n - h)
\phi(S_n)} \, \ind_{\{S_N=0\}} \right], \qquad N\in\N_0,
\end{equation}
the existence of the limit 
\begin{equation}
\lim_{N\to\infty} \frac{1}{N} \log Z_{N,\beta,h}^{\omega,\rc} 
\quad \text{$\bbP$-a.s.\ and in $L^1(\bbP)$}
\end{equation}
follows by standard super-additivity arguments. Since $\phi$ is bounded and $\bbP$ has 
finite exponential moments (recall \eqref{eq:expmom}), the limit is finite. We will show in 
Appendix~\ref{app:compare} that 
\begin{equation}
\label{eq:compare}
\lim_{N\to\infty} \frac{1}{N} \log \frac{Z_{N,\beta,h}^{\omega,\rc}}{Z_{N,\beta,h}^\omega} = 0
\quad \text{$\bbP$-a.s.\ and in $L^1(\bbP)$},
\end{equation} 
so that \eqref{eq:f} follows.

By \eqref{eq:phidecay}, for every $\epsilon > 0$ there is an $M\in\N_0$ 
such that $0 \leq \phi(x) \leq \epsilon$ for $x \geq M$. Therefore
\begin{equation}
Z_{N,\beta,h}^{\omega,\rc} \geq \ee^{-\|\phi\|_\infty\sum_{n=1}^{M-1} 
(\beta|\omega_n|+|h|) - \epsilon \sum_{n=M}^N (\beta|\omega_n|+|h|)} \,
\p(S_n \geq M\,\,\forall\,M < n \leq N) \quad\forall\,M\in\N_0.
\end{equation}
We will show in Appendix~\ref{app:exceed} that, by \eqref{eq:astau}, 
\begin{equation}
\label{eq:exceed}
\lim_{N\to\infty} \frac{1}{N} \log \p(S_n \geq M\,\,\forall\,M < n \leq N) = 0 \quad\forall\,M\in\N_0,
\end{equation}
and so it follows that $\f^\mathrm{que}(\beta,h) \geq - \epsilon (\beta \bbE[|\omega_1|]+|h|)$. 
Since $\epsilon > 0$ is arbitrary, we obtain the important inequality
\begin{equation} 
\label{eq:fege0}
\f^\que(\beta,h) \geq 0  \qquad \forall\, \beta \in [0,\infty), h \in \R.
\end{equation}
It is therefore natural to define the two phases
\begin{equation}
\begin{aligned}
\cL^\que &:= \{(\beta,h)\colon\, \f^\que(\beta,h) > 0\},\\
\cD^\que &:= \{(\beta,h)\colon\, \f^\que(\beta,h) = 0\},
\end{aligned}
\end{equation}
which we refer to as the \emph{quenched localized phase}, respectively, the 
\emph{quenched delocalized phase}.

By \eqref{eq:f}, $(\beta,h) \mapsto \f^\que(\beta,h)$ is the pointwise limit of jointly convex 
functions. Moreover, $h \mapsto Z_{N,\beta,h}^\omega$ is non-increasing, so that 
$h \mapsto \f^\que(\beta,h)$ is non-increasing as well. Furthermore, $\beta \mapsto 
\bbE[\log Z_{N,\beta,h}^\omega]$ is convex and (by direct computation) has zero 
derivative at $\beta = 0$, so that $\beta \mapsto \f^\que(\beta,h)$ is non-decreasing on 
$[0,\infty)$.

%%%%%%%%%%%%%%%%%%%%%%%%%%%%%%%%%%%%%%%%%
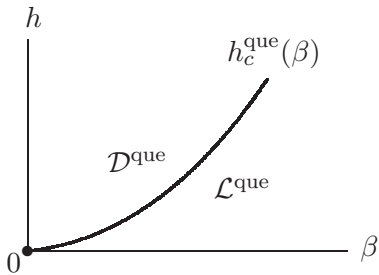
\begin{figure}[htbp]
\begin{center}
\setlength{\unitlength}{0.35cm}
\begin{picture}(12,12)(0,-1.5)
\put(0,0){\line(12,0){12}}
\put(0,0){\line(0,8){8}}
{\thicklines
\qbezier(0,0)(5,0.5)(9,6.5)
}
\put(-.8,-.8){$0$}
\put(12.5,-0.2){$\beta$}
\put(-0.1,8.5){$h$}
\put(0,0){\circle*{.4}}
\put(7,2){$\cL^\que$}
\put(3,3){$\cD^\que$}
\put(7.5,7.2){$h_c^\que(\beta)$}
\end{picture}
\end{center}
\vspace{0cm}
\caption{\small Qualitative plot of $\beta \mapsto h_c^\que(\beta)$.}
\label{fig-critcurveque}
\end{figure}
%%%%%%%%%%%%%%%%%%%%%%%%%%%%%%%%%%%%%%

From the monotonicity of $h \mapsto \f^\que(\beta,h)$ it follows that $\cL^\que$ 
and $\cD^\que$ are separated by a \emph{quenched critical curve} $h_c^\que
\colon\, [0,\infty) \to [0,\infty)$ whose graph is $\partial\cD^\que$ (see Fig.~\ref{fig-critcurveque}):
\begin{equation}
\begin{aligned}
\cL^\que &:= \{(\beta,h)\colon\, h < h_c^\que(\beta)\},\\
\cD^\que &:= \{(\beta,h)\colon\, h \geq h_c^\que(\beta)\}.
\end{aligned}
\end{equation}
From the convexity of $(\beta,h) \mapsto \f^\que(\beta,h)$ it follows that the lower level set 
$\cD^\que = \{(\beta,h)\colon\, \f^\que(\beta,h) \leq 0\}$ is convex. Since $\cD^\que$ is the 
upper graph of $h_c^\que$, it follows that $h_c^\que$ is convex and hence continuous. 
In Section~\ref{ss:thmsgen} we will see that $h_c^\que$ is finite everywhere. Since $S$ 
is recurrent, it follows from the theory of the homogeneous pinning model (Giacomin~\cite{cf:Gi1}, 
\cite{cf:Gi2}, den Hollander~\cite{cf:dHo}) that $\f^\que(0,h) > 0$ for $h < 0$ and $\f^\que(0,h) 
= 0$ for $h \geq 0$. Hence $h_c^\que(0) = 0$. (Note that $\phi(x) \geq \phi(x_*) \ind_{\{x_*\}}(x)$ 
for any $x_*\in\Z$ with $\phi(x_*) > 0$. Therefore we can dominate the quenched free energy 
for $\beta = 0$ by the free energy of the homogeneous pinning model with a strictly positive 
pinning reward.)

Finally, from the monotonicity of $\beta \mapsto \f^\que(\beta,h)$ on $[0,\infty)$ it follows 
that $\f^\que(\beta,h) \geq \f^\que(0,h) > 0$ for $h < 0$, so that $h_c^\que(\beta) \geq 
h_c^\que(0) = 0$ for $\beta \geq 0$. Since $h_c^\que$ is convex, this implies that $h_c^\que$ 
is non-decreasing, and is strictly increasing as soon as it leaves zero. In Section~\ref{ss:thmsgen} 
we will see that $h_c^\que(\beta) > 0$ for all $\beta>0$ (see Fig.~\ref{fig-critcurveque}).

%%%%%%%%%%%%%%%%%%%

\subsection{The annealed free energy}
\label{ss:annfe}

The \emph{annealed partition function} associated with \eqref{eq:model1} is
\begin{equation} 
\label{eq:Zann}
Z_{N, \beta, h}^\ann := \bbE [Z_{N,\beta,h}^\omega] 
= \e\left[ \ee^{\sum_{n=1}^N \psi_{\beta,h}(S_n)}\right], \qquad N\in\N_0,
\end{equation}
where
\begin{equation}
\label{eq:psi}
\psi_{\beta,h}(x) := \log \M(\beta \phi(x)) - h \phi(x).
\end{equation}
This is the partition function of the homogeneous pinning model with potential $\psi_{\beta,h}$. 
A delicate point is that $\psi_{\beta,h}$ does not have a sign: it may be a mixture of attractive 
and repulsive interactions. This comes from the fact that both the charge distribution and the 
interaction potential are general.

The \emph{annealed free energy} is defined by
\begin{equation}
\label{eq:fann}
\f^\ann(\beta,h) := \lim_{N\to\infty} \frac{1}{N} \log Z_{N, \beta, h}^\ann.
\end{equation}
For the \emph{constrained} partition function
\begin{equation} 
\label{eq:Zcann}
Z_{N,\beta,h}^{\ann,\rc} := \e\left[ \ee^{\sum_{n=1}^N \psi_{\beta,h}(S_n)} 
\, \ind_{\{S_N=0\}} \right], \qquad N\in\N_0,
\end{equation}
the existence of the limit 
\begin{equation}
\lim_{N\to\infty} \frac{1}{N} \log Z_{N,\beta,h}^{\ann,\rc} 
\end{equation}
again follows by standard super-additivity arguments. Since $\psi_{\beta,h}$ is bounded,
the limit is finite. The analogue of \eqref{eq:compare}, which will be proved in 
Appendix~\ref{app:compareann}, reads
\begin{equation}
\label{eq:compareann}
\lim_{N\to\infty} \frac{1}{N} \log \frac{Z_{N,\beta,h}^{\ann,\rc}}{Z_{N,\beta,h}^\ann} = 0,
\end{equation} 
so that \eqref{eq:fann} follows.

The \emph{annealed localized phase}, \emph{annealed delocalized phase} and \emph{annealed
critical curve} are defined as
\begin{equation}
\begin{aligned}
\cL^\ann &:= \{(\beta,h)\colon\, \f^\ann(\beta,h) > 0\} 
= \{(\beta,h)\colon\, h < h_c^\ann(\beta)\}, \\
\cD^\ann &:= \{(\beta,h)\colon\, \f^\ann(\beta,h) = 0\} 
= \{(\beta,h)\colon\, h \geq h_c^\ann(\beta)\}.
\end{aligned}
\end{equation}
As is clear from \eqref{eq:psi} and the fact that $\phi$ is non-negative, $\f^\ann(\beta,h)$ is 
non-increasing and convex as a function of $h$, and non-decreasing as a function of $\beta$ 
but not necessarily convex. Later we will see that nonetheless $\beta \mapsto h_c^\ann(\beta)$ 
has a shape that is qualitatively similar to that of $\beta \mapsto h_c^\que(\beta)$ (see 
Fig.~\ref{fig-critcurveann}).
 
An important property of the annealed free energy is that it provides an upper bound for 
the quenched free energy: by Jensen's inequality we have $\f^\que(\beta,h) \leq \f^\ann(\beta,h)$ 
for all $\beta \in [0,\infty)$ and $h\in\R$. Recalling \eqref{eq:fege0}, we therefore see that
\begin{equation}
\label{eq:annbound}
0 \leq h_c^\que(\beta) \leq h_c^\ann(\beta) \qquad \forall\, \beta \geq 0.
\end{equation}
Unlike for the pinning model and the copolymer model, for general potentials $\phi$
\emph{the annealed free energy and the annealed critical curve are not known explicitly}.

%%%%%%%%%%%%%%%%%%%%%%%%%%%%%%%%%%%%%%%%%
\begin{figure}[htbp]
\begin{center}
\setlength{\unitlength}{0.35cm}
\begin{picture}(12,12)(0,-1.5)
\put(0,0){\line(12,0){12}}
\put(0,0){\line(0,8){8}}
{\thicklines
\qbezier(0,0)(5,0.5)(9,6.5)
}
\put(-.8,-.8){$0$}
\put(12.5,-0.2){$\beta$}
\put(-0.1,8.5){$h$}
\put(0,0){\circle*{.4}}
\put(7,2){$\cL^\ann$}
\put(3,3){$\cD^\ann$}
\put(7.5,7.2){$h_c^\ann(\beta)$}
\end{picture}
\end{center}
\vspace{0cm}
\caption{\small Qualitative plot of $\beta \mapsto h_c^\ann(\beta)$.}
\label{fig-critcurveann}
\end{figure}
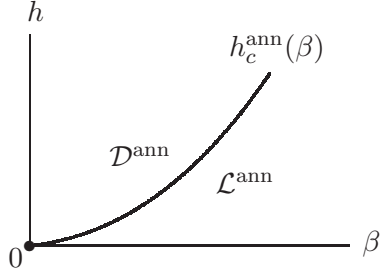
%%%%%%%%%%%%%%%%%%%%%%%%%%%%%%%%%%%%%%

%%%

\subsection{Scaling properties of the Bessel random walk}
\label{ss:Besselscal}

Part of our results below involve the annealed free energy and the annealed critical 
curve associated with a Brownian version of the model, where the reference measure 
is based on the \emph{Bessel process} $X := (X_t)_{t \geq 0}$ of dimension $2(1-\alpha)$
(see \cite[Chapter~XI]{cf:RevYor}) defined by
\begin{equation}
\label{eq:BP}
\dd X_t = \dd B_t - \frac{\alpha-\tfrac12}{X_t}\,\dd t \quad \text{ on } [0,\infty)
\text{ with reflection at } 0,
\end{equation}
where $\alpha \in (0,1)$ and $(B_t)_{t \geq 0}$ is standard Brownian motion on $\R$.\footnote{Formally, 
the \emph{squared Bessel process} $(Y_t)_{t \geq 0}$ is defined by $\dd Y_t = 2\sqrt{Y_t}\,\dd B_t 
+ 2(1-\alpha)\,\dd t$ on $[0,\infty)$ with reflection at $0$, and $(X_t)_{t \geq 0}$ is defined 
by setting $X_t=\sqrt{Y_t}$.} Informally, $X$ makes infrequent visits to $0$ when $\alpha<\tfrac12$ 
and frequent visits to $0$ when $\alpha>\tfrac12$. The choice $\alpha=\tfrac12$ corresponds 
to reflected Brownian motion, i.e., $X_t = |B_t|$ with $(B_t)_{t \ge 0}$ standard Brownian motion.
We write $\hat{\p}_x$ to denote the law of $X$ given $X_0 = x$. When $x=0$, we simply write 
$\hat{\p}=\hat{\p}_0$.

For the semigroup $g_t(x,y) := \hat{\p}_x(X_t \in \dd y )/\dd y$ of $X$ there is an explicit formula, 
namely,
\begin{equation} 
\label{eq:gtxy}
g_t(x,y) = \frac{x^\alpha \, y^{1-\alpha}}{t} \, \ee^{-\frac{x^2 + y^2}{2t}}
\, J_{-\alpha}\bigg(\frac{xy}{t}\bigg), \qquad x, y \in (0,\infty), \, t > 0,
\end{equation}
where $J_{-\alpha}(z) = \sum_{m\in\N_0} \frac{(-1)^m}{m! \, \Gamma(m+1-\alpha)} (\frac{z}{2})^{2m-\alpha}$ 
is the Bessel function of index $-\alpha$. Since
\begin{equation}
J_{-\alpha}(z) \sim \frac{2^\alpha}{\Gamma(1-\alpha)} z^{-\alpha}, \qquad z \downarrow 0,
\end{equation}
we also have the explicit formula
\begin{equation} 
\label{eq:gt}
g_t(y) := g_t(0,y) := \frac{\hat{\p}(X_t \in \dd y)}{\dd y} = \frac{2^\alpha}{\Gamma(1-\alpha)}
\,\frac{y^{1-2\alpha}}{t^{1-\alpha}} \, \ee^{-\frac{y^2}{2t}}, \qquad y \in [0,\infty), \, t > 0.
\end{equation}
It follows that
\begin{equation} 
\label{eq:PXepsilon}
\hat\p(X_t < \epsilon) \le \frac{\epsilon^{2(1-\alpha)}}{c_\alpha \, t^{1-\alpha}}, \quad t, \epsilon > 0,
\qquad c_\alpha := \frac{\Gamma(2-\alpha)}{2^{\alpha-1}},
\end{equation}
and this inequality becomes \emph{sharp} as $\epsilon \downarrow 0$, namely,
\begin{equation}
\label{eq:hatgt0}
\lim_{\epsilon \downarrow 0} 
\frac{c_\alpha}{\epsilon^{2(1-\alpha)}} \, \hat{\p}_x(X_t < \epsilon)
= \lim_{\epsilon \downarrow 0} 
\frac{c_\alpha}{\epsilon^{2(1-\alpha)}}  \, \int_0^\epsilon g_t(x,y) \, \dd y
= \hat g_t(x,0), \qquad x \geq 0,
\end{equation}
where
\begin{equation}
\label{eq:hatgt}
\hat g_t(x,0) := \frac{1}{t^{1-\alpha}} \, \ee^{-\frac{x^2}{2t}}, \qquad
x \in [0,\infty), \, t > 0.
\end{equation}

The \emph{local time} of $X$ at $0$ up to time $T \ge 0$ is defined as the following limit
in probability:
\begin{equation}
\label{eq:hatLdef}
\hat{L}_T(0) := \lim_{\epsilon \downarrow 0} \frac{c_\alpha}{\epsilon^{2(1-\alpha)}} 
\int_0^T \dd t\,\,1_{(0,\epsilon)}(X_t). 
\end{equation}
We will informally write
\begin{equation}
\label{eq:hatdelta}
\hat{L}_T(0) =: \int_0^T \dd t\,\hat\delta_0(X_t).
\end{equation}
Note that $\hat{\e}[\hat{L}_T(0)] = \int_0^T \frac{\dd t}{t^{1-\alpha}}= \int_0^T \hat 
g_t(0,0) \, \dd t = \frac{1}{\alpha} \, T^{\alpha}$. 

The relation of the Bessel process with the Bessel random walk defined in Remark~\ref{rem:Bessel} is 
that the latter satisfies the invariance principle (see Lamperti~\cite{cf:Lam})
\begin{equation}
\label{eq:invprin}
\big(|S_{\lfloor Nt \rfloor}|/\sqrt{N}\,\big)_{t\geq 0} \Longrightarrow (X_t)_{t \geq 0}, 
\qquad N\to\infty.
\end{equation}
Write $\p_{k'}(\,\cdot\,) := \p(\,\cdot\,|\,S_0 = k')$ to denote the law of the Bessel random
walk started at $k'\in\Z$, so that $\p = \p_0$. Local limit theorems for the transition probabilities 
$\p_{k'}(|S_n|=k)$ of the Bessel random walk have been established in \cite[Theorem~2.4]{cf:Al}. 
The following formulas hold in the limit as $n\to\infty$, uniformly in a specified range of $k, k' \in\N_0$. 
We assume that $k-k'$ has the same parity as $n$, i.e., $k-k'-n$ is even, because otherwise 
$\p_{k'}(|S_n|=k)$=0.
\begin{itemize}
\item 
\emph{Low ending heights:} 
For any $\epsilon > 0$ and $\bar k_n = o(\sqrt{n}\,)$, uniformly in $0 \le k' \le \sqrt{n}/\epsilon$ and 
$0 \le k \le \bar k_n$,
\begin{equation}
\label{eq:Besseltail}
\p_{k'}(|S_n| = k) \sim  \frac{2\,c(k)}{n^{1-\alpha}}\,
\hat g_1\bigg(\frac{k'}{\sqrt{n}},0\bigg) \,
\ind_{\{\text{$k-k'-n$ is even}\}}, \qquad n \to \infty,
\end{equation}
where $\hat g_1$ is defined in \eqref{eq:hatgt} and $c\colon\, \N_0 \to (0,\infty)$ is an explicit function 
(which depends on the function $d\colon\,\Z \to \R$ in \eqref{eq:Besseltr}) such that
\begin{equation}
\label{eq:ctildeasymp}
c(k) \sim \frac{2^{\alpha}}{\Gamma(1-\alpha)}\,k^{1-2\alpha}, \qquad k \to \infty.
\end{equation}
In case of a low starting height $k' = o(\sqrt{n}\,)$, \eqref{eq:Besseltail} simplifies because 
$\hat g_1(\frac{k'}{\sqrt{n}},0) \sim \hat g_1(0,0) = 1$.
\item 
\emph{Intermediate ending height:} 
For any $\epsilon > 0$, uniformly in $0 \le k' \le \sqrt{n}/\epsilon$ and $\epsilon\sqrt{n} \le k 
\le \sqrt{n}/\epsilon$,
\begin{equation} 
\label{eq:llt}
\p_{k'}(|S_n| = k) 
\sim \frac{2}{\sqrt{n}} \, g_1\bigg(\frac{k'}{\sqrt{n}, }\frac{k}{\sqrt{n}}\bigg)\,
\ind_{\{\text{$k-k'-n$ is even}\}}, 
\qquad n \to \infty,
\end{equation}
where $g_1$ is the density in \eqref{eq:gtxy} at time $1$. In case of a low starting height 
$k' = o(\sqrt{n}\,)$, \eqref{eq:Besseltail} simplifies because $g_1(\frac{k'}{\sqrt{n}},\frac{k}{\sqrt{n}}) 
\sim g_1(0,\frac{k}{\sqrt{n}}) = g_1(\frac{k}{\sqrt{n}})$ reduces to \eqref{eq:gt}.
\item
\emph{High ending heights:} There exists a $C = C(\alpha) < \infty$ such that, for all $k \ge \sqrt{n}$,
\begin{equation}
\label{eq:lltub}
\p(|S_n| = k) \le  \frac{C}{\sqrt{n}} \, \ee^{-\frac{k^2}{8n}}\,\ind_{\{\text{$k-n$ is even}\}}.
\end{equation}
\end{itemize}
It follows from \eqref{eq:Besseltail}--\eqref{eq:lltub} that, for some $C < \infty$,
\begin{equation} 
\label{eq:globalUB}
\forall\, n \in \N \quad \forall\, k \in \Z\colon \qquad
\p(|S_n| = k) \le C \, \frac{(1+|k|)^{1-2\alpha}}{n^{1-\alpha}} \, e^{-\frac{k^2}{8n}}
\,\ind_{\{\text{$k-n$ is even}\}}.
\end{equation}
This uniform bound will be needed to control scaling computations.

\begin{remark}
{\rm Equations \eqref{eq:Besseltail} and \eqref{eq:llt} for $k' = 0$ are proved in \cite[Theorem~2.4]{cf:Al}, while 
the case $k' \neq 0$ follows via the relation $\p_{k'}(|S_n| = k) = \sum_{m=1}^n \p_{k'}(\tau_1 = m) \, \p_0(|S_{n-m}|=k) 
+ O(\p_{k'}(\tau_1 > n))$ (see also the estimates on $\p_{k'}(\tau_1 = m)$ provided in \cite[Theorem~2.2]{cf:Al}).
We further point out the duality relation $\p_{k'}(|S_n| = k) = \frac{1+d(k')}{1+d(k)} \frac{\lambda_{k'}}{\lambda_k} 
\, \p_k(|S_n| = k')$, where $\lambda_k := \prod_{x=1}^k \frac{1-d(x)}{1+d(x)}$.}
\end{remark}

%%%%%%%%%%%%%%%%%%%%

\subsection{General properties}
\label{ss:thmsgen}

Our first set of theorems concerns the quenched and the annealed critical curve.

\begin{theorem}
\label{thm:varrep}
$\beta \mapsto h_c^\que(\beta)$ and $\beta \mapsto h_c^\ann(\beta)$ can be characterized 
in terms of variational formulas (see Theorems~{\rm \ref{thm:annchar}--\ref{thm:quechar}} 
below).  
\end{theorem}

\begin{theorem}
\label{thm:boundshc}
For every $\beta \geq 0$,
\begin{equation}
\label{eq:boundshc}
(1+\alpha)\, h_c^\ann \bigg(\frac{\beta}{1+\alpha} \bigg) 
\leq h_c^\que(\beta) \leq h_c^\ann(\beta).
\end{equation}
\end{theorem}

\noindent
As already noted in \eqref{eq:annbound}, the second inequality in \eqref{eq:boundshc} is an 
immediate consequence of Jensen's inequality applied to \eqref{eq:Zc} and \eqref{eq:Zcann}. The first
inequality in \eqref{eq:boundshc}, which is known as the Monthus-Bodineau-Giacomin bound, was previously 
shown to hold for the copolymer model~\cite{cf:BG}, \cite{cf:BodHoOp}. We show that it holds for the general 
class of potentials satisfying \eqref{eq:phidecay}.

In Section~\ref{ss:applvarcrit} we will show that $F^\ann(\beta,0)>0$ for all $\beta>0$. This 
implies that $h_c^\ann(\beta)>0$ for all $\beta>0$, which via \eqref{eq:boundshc} settles the 
claim made at the end of Section~\ref{ss:quefe} that $h_c^\que(\beta)>0$ for all $\beta>0$.

%%%%

\subsection{Scaling for weak interaction}
\label{ss:thmsweak}

Our second set of theorems looks at the scaling of the annealed free energy in the limit of weak interaction, 
for the special case where $S$ is the Bessel random walk with parameter $\alpha \in (0,1)$ defined in 
Remark~\ref{rem:Bessel} and the interaction potential $\phi$ is symmetric:
$\phi(-x) = \phi(x)$ for all $x\in\Z$. We consider \emph{three different regimes} for the tail 
behaviour of $\phi$, namely, 
\begin{equation}
\label{eq:phiscal}
\lim_{|x|\to\infty} |x|^\theta\phi(x) = c \in (0,\infty)
\end{equation}
with 
\begin{equation}
\label{eq:regimes}
\begin{aligned}
&\theta \in (0,1-\alpha),\\
&\theta \in (1-\alpha,2(1-\alpha)),\\
&\theta \in (2(1-\alpha),\infty).
\end{aligned}
\end{equation}
  
\begin{theorem}
\label{thm:weak1}
Suppose that $\alpha \in (0,1)$ and $\theta \in (0,1-\alpha)$. For every $\hat\beta \in (0,\infty)$ and 
$\hat{h} \in (0,\infty)$, 
\begin{equation}
\label{eq:convweak1}
\lim_{\delta \downarrow 0}\, \delta^{-1}\,F^\ann\left(\hat\beta\,\delta^{(1-\theta)/2},
\hat{h}\,\delta^{(2-\theta)/2}\right) = \hat{F}^\ann(\hat\beta,\hat{h}),
\end{equation}
where 
\begin{equation}\label{eq:hatF1}
\hat{F}^\ann(\hat\beta,\hat{h}) := \lim_{T\to\infty} \frac{1}{T} \log \hat{\e}\left[\exp\left(
\tfrac12\hat\beta^2c^2 \int_0^T \dd t\,\,X_t^{-2\theta}
-\hat{h}c \int_0^T \dd t\,\,X_t^{-\theta}\right)\right]
\end{equation}
with $c$ the constant in \eqref{eq:phiscal}.
\end{theorem}

\begin{theorem}
\label{thm:weak2}
Suppose that $\alpha \in (0,1)$ and $\theta \in (1-\alpha,2(1-\alpha))$. For every $\hat\beta \in (0,\infty)$ 
and $\hat{h} \in (0,\infty)$, 
\begin{equation}
\label{eq:convweak2}
\lim_{\delta \downarrow 0}\, \delta^{-1}\,F^\ann\left(\hat\beta\,\delta^{\alpha/2},
\hat{h}\,\delta^{(2-\theta)/2}\right) = \hat{F}^\ann(\hat\beta,\hat{h}),
\end{equation}
where
\begin{equation}\label{eq:hatF2}
\hat{F}^\ann(\hat\beta,\hat{h}) := \lim_{T\to\infty} \frac{1}{T} \log \hat{\e}\left[\exp\left(
\tfrac12\hat\beta^2\,c^*[\phi^2]\,\hat{L}_T(0)
- \hat{h}c \int_0^T \dd t\,\,X_t^{-\theta}\right)\right]
\end{equation}
with $c$ the constant in \eqref{eq:phiscal} and $c^*[\phi^2]=\sum_{x\in\Z} \phi^2(x)\,c(x)$,
where $x \mapsto c(x) := c(|x|)$ is the function in \eqref{eq:Besseltail}--\eqref{eq:ctildeasymp}.
\end{theorem}

\begin{theorem}
\label{thm:weak3}
Suppose that $\alpha \in (0,1)$ and $\theta \in (2(1-\alpha),\infty)$. For every $\hat\beta \in (0,\infty)$
and $\hat{h} \in (0,\infty)$, 
\begin{equation}
\label{eq:convweak3}
\lim_{\delta \downarrow 0}\, \delta^{-1}\,F^\ann\left(\hat\beta\,\delta^{\alpha/2},
\hat{h}\,\delta^{\alpha}\right) = \hat{F}^\ann(\hat\beta,\hat{h}),
\end{equation}
where
\begin{equation}\label{eq:hatF3}
\hat{F}^\ann(\hat\beta,\hat{h}) := \lim_{T\to\infty} \frac{1}{T} \log \hat{\e}\left[\exp\left(
\Big\{\tfrac12\hat{\beta}^2\,c^*[\phi^2]-\hat{h}\,c^*[\phi]\Big\}\,\hat{L}_T(0)\right)\right]
\end{equation}
with $c^*[\phi] = \sum_{x\in\Z} \phi(x) c(x)$, where $x \mapsto c(x) := c(|x|)$ is the function in 
\eqref{eq:Besseltail}--\eqref{eq:ctildeasymp}.
\end{theorem}

Note that, because of \eqref{eq:ctildeasymp} and \eqref{eq:phiscal}, $c^*[\phi^2]<\infty$ when 
$\theta>1-\alpha$ and $c^*[\phi]<\infty$ when $\theta>2(1-\alpha)$. In  Appendix~\ref{app:finpartcont} 
we will show that the annealed partition functions associated with the Bessel process appearing 
in Theorems~\ref{thm:weak1}--\ref{thm:weak3} are finite, and so are the corresponding 
annealed free energies. 

%%%%%%%%%%%%%%%%%%%%%%%%%%%%%%%%%%%%%%%%%
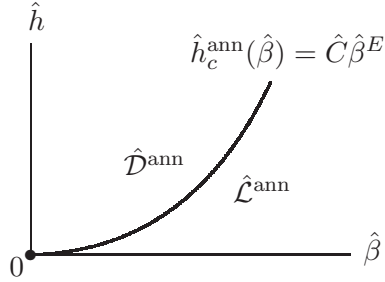
\begin{figure}[htbp]
\begin{center}
\setlength{\unitlength}{0.35cm}
\begin{picture}(12,12)(0,-1.5)
\put(0,0){\line(12,0){12}}
\put(0,0){\line(0,8){8}}
{\thicklines
\qbezier(0,0)(6,0.1)(9,6.5)
}
\put(-.8,-.8){$0$}
\put(12.5,-0.2){$\hat\beta$}
\put(-0.1,8.5){$\hat{h}$}
\put(0,0){\circle*{.4}}
\put(7.6,2){$\hat\cL^\ann$}
\put(3.5,3){$\hat\cD^\ann$}
\put(6,7.3){$\hat{h}_c^\ann(\hat\beta)=\hat{C}\hat\beta^E$}
\end{picture}
\end{center}
\vspace{0cm}
\caption{\small Plot of $\hat\beta \mapsto \hat{h}_c^\ann(\hat\beta)$ in \eqref{eq:contcritcurve}.}
\label{fig-critcurveanncont}
\end{figure}
%%%%%%%%%%%%%%%%%%%%%%%%%%%%%%%%%%%%%%

The annealed free energy $\hat{F}^\ann(\hat\beta,\hat{h})$ appearing in Theorems~\ref{thm:weak1}--\ref{thm:weak3}
has its own phase diagram, with phases 
\begin{equation}
\begin{aligned}
\hat\cL^\ann &:= \{(\hat{\beta},\hat{h})\colon\,\hat{F}^\ann(\hat\beta,\hat{h})>0\},\\
\hat\cD^\ann &:= \{(\hat{\beta},\hat{h})\colon\,\hat{F}^\ann(\hat\beta,\hat{h})=0\},
\end{aligned}
\end{equation}
and with a critical curve that is a perfect power law (see Fig.~\ref{fig-critcurveanncont}),
namely,
\begin{equation}
\label{eq:contcritcurve}
\hat{h}_c^\ann(\hat\beta) = \hat{C}\hat\beta^E, \qquad \hat\beta \in (0,\infty), 
\end{equation}
where 
\begin{equation}
\label{eq:exponent}
E=E(\alpha,\theta) = \left\{\begin{array}{ll}
(2-\theta)/(1-\theta), &\theta \in (0,1-\alpha),\\[0.1cm]
(2-\theta)/\alpha, &\theta \in (1-\alpha,2(1-\alpha)),\\[0.1cm]
2, &\theta \in (2(1-\alpha),\infty),
\end{array}
\right.
\end{equation} 
plays the role of a \emph{critical exponent} (see Fig.~\ref{fig-exponent}). The scaling 
of the annealed critical curve in Theorems~\ref{thm:weak1}--\ref{thm:weak3} can
be summarised as saying that $h_c^\ann(\beta) \sim \hat{h}_c^\ann(\beta)$, $\beta
\downarrow 0$.

%%%%%%%%%%%%%%%%%%%%%%%%%%%%%%%%%%%%%%%%%
\begin{figure}[htbp]
\begin{center}
\setlength{\unitlength}{0.5cm}
\begin{picture}(12,12)(0,-4.5)
\put(0,0){\line(12,0){12}}
\put(0,0){\line(0,6){6}}
{\thicklines
\qbezier(0,2)(2,2.2)(4,4)
\qbezier(4,4)(6,3)(8,2)
\qbezier(8,2)(9,2)(10,2)
}
\qbezier[50](0,2)(5,2)(11,2)
\qbezier[50](0,4)(5,4)(11,4)
\qbezier[30](4,0)(4,2)(4,4)
\qbezier[15](8,0)(8,1)(8,2)
\put(-.2,-1){$0$}
\put(12.5,-0.2){$\theta$}
\put(-0.3,6.5){$E$}
\put(4,4){\circle*{.3}}
\put(8,2){\circle*{.3}}
\put(-0.8,1.8){$2$}
\put(-1.6,3.8){$\frac{1+\alpha}{\alpha}$}
\put(3.2,-1){{\small $1-\alpha$}}
\put(6.6,-1){{\small $2(1-\alpha)$}}
\end{picture}
\end{center}
\vspace{-1.5cm}
\caption{\small Plot of the critical exponent $E$ in \eqref{eq:exponent} as a function
of $\theta$ for fixed $\alpha$. The three regimes for $\theta$ are indicated. No 
information is available at the two crossover points.}
\label{fig-exponent}
\end{figure}
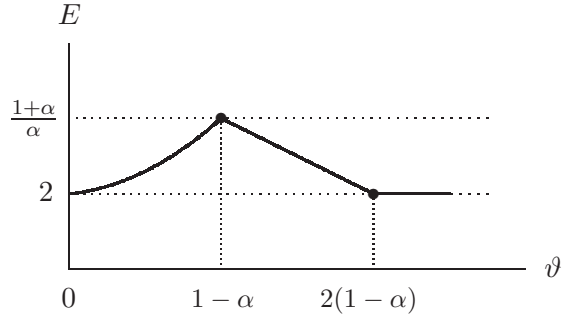
%%%%%%%%%%%%%%%%%%%%%%%%%%%%%%%%%%%%%%
 
The constant $\hat{C}$ depends on $\alpha,\phi$ and can be characterized as the unique solution 
$\hat{C} \in (0,\infty)$ of the equation $\hat{F}^\ann(1,\hat{C})=0$. This constant is hard to identify 
in the first two regimes. In the third regime $\theta \in (2(1-\alpha),\infty)$ it is found by inserting 
\eqref{eq:contcritcurve} into the equation 
\begin{equation}
\Delta\big(\hat{\beta},\hat{h}^\ann_c(\hat\beta)\big) = 0 \quad \text{ with } \quad
\Delta(\hat{\beta},\hat{h}) := \tfrac12\hat{\beta}^2\,c^*[\phi^2]-\hat{h}\,c^*[\phi],
\end{equation} 
which gives
\begin{equation}
\hat{C} = \frac{c^*[\phi^2]}{2 c^*[\phi]}.
\end{equation}
We show in Appendix~\ref{app:annfeexplicit} that, for the the third regime $\theta \in (2(1-\alpha),\infty)$,
the annealed free energy $\hat{F}^\ann(\hat{\beta},\hat{h})$ can be computed explicitly, namely,
\begin{equation}
\label{eq:Fannid}
\hat{F}^\ann(\hat{\beta},\hat{h}) 
= \Big( \Gamma(\alpha) \, \big[0 \vee \Delta(\hat{\beta},\hat{h})\big] \Big)^{1/\alpha}.
\end{equation} 

\begin{remark}
{\rm In view of the scaling limit for the annealed free energies described in Theorems~\ref{thm:weak1}--\ref{thm:weak3},
it is natural to expect a scaling limit for the corresponding annealed critical curves as well. Indeed, 
the continuum critical curve is the perfect power law in \eqref{eq:contcritcurve}, where $\hat{C}$ and $E$ 
depend on $\alpha$ and $\phi$ (and hence on $\theta$). We \emph{conjecture} that \eqref{eq:contcritcurve} 
captures the asymptotic behaviour for weak interaction of the \emph{discrete} critical curve $h_c^\ann(\beta)$ as well, 
in the sense that in all three regimes we should have
\begin{equation} 
\label{eq:thefull}
\lim_{\beta \downarrow 0} \beta^{-E}\,h_c^\ann(\beta) = \hat{C}.
\end{equation}
This scaling relation cannot be simply deduced from Theorems~\ref{thm:weak1}--\ref{thm:weak3}, because 
pointwise convergence of the free energies does not imply convergence of their zero-level sets, of which 
the critical curves are the boundaries. However, half of \eqref{eq:thefull} follows because if the continuum 
free energy is strictly positive, then the rescaled discrete free energy eventually becomes strictly positive too
in the weak interaction limit, which leads to
\begin{equation}
\liminf_{\beta \downarrow 0} \beta^{-E}\,h_c^\ann(\beta) \ge \hat{C} .
\end{equation}
In order to prove \eqref{eq:thefull} extra work is needed: the scaling of the free energies in 
\eqref{eq:convweak1}, \eqref{eq:convweak2} and \eqref{eq:convweak3} must be strengthened to 
a \emph{perturbative scaling}, as shown in \cite{cf:BodHo} and \cite{cf:CaGi} for the copolymer model
and in \cite{cf:CaToTo} for the pinning model.}
\end{remark}

%%%%%%%%%%%%%%%%%%%%%

\subsection{Discussion}
\label{ss:disc}

We comment on the results in Sections~\ref{ss:thmsgen}--\ref{ss:thmsweak}.

\medskip\noindent
{\bf 1.} 
The results in Theorems~\ref{thm:varrep}--\ref{thm:boundshc} are known for the
special case where the interaction potential is that of the pinning model or the 
copolymer model defined in \eqref{eq:phipincop}. However, the techniques used 
for these two cases do not carry over to the general class of potentials considered 
in \eqref{eq:phidecay}. Intuitively, the reason why extension is possible is that the 
conditions stated in \eqref{eq:phidecay} say that, outside a large interval around 
the origin in $\Z$, the interaction potential is controlled by a multiple of that 
of the copolymer model.   

\medskip\noindent
{\bf 2.} 
As we will see in Section~\ref{s:varcrit}, the variational formula for $h_c^\que(\beta)$ 
mentioned in Theorem~\ref{thm:varrep} involves a supremum over the space of all 
shift-invariant probability distributions on the set of infinite sequences of words of 
arbitrary length drawn from an infinite sequence of letters taking values in $\R \times \Z$. 
The supremum involves a quenched rate function that captures the \emph{complexity 
of the interplay} between the disorder of the charges and the excursions of the polymer 
away from the interface. This variational formula is hard to manipulate, but it is the 
starting point for the proof of Theorems~\ref{thm:boundshc}. The variational formula 
for $h_c^\ann(\beta)$ mentioned in Theorem~\ref{thm:varrep} is simpler, but still not 
easy to manipulate (see \eqref{eq:condloc} below).  

\medskip\noindent
{\bf 3.}
Note that for $h = 0$ and $\beta > 0$ the annealed partition function $Z_{N, \beta, h}^\ann$ 
is bounded from below by the partition function of a homogenous pinning model with a 
strictly positive reward, which is localized. The lower bound in Theorem~\ref{thm:boundshc}
therefore shows that $h^\que_c(\beta) > 0$ for every $\beta > 0$. Since $\beta \mapsto 
h_c^\que(\beta)$ is convex, it must therefore be strictly increasing (see Fig.~\ref{fig-critcurveque}).

\medskip\noindent
{\bf 4.} 
Disorder has a tendency to smoothen the phase transition. In Caravenna and den 
Hollander~\cite{cf:CadHo} a general \emph{smoothing inequality} is derived that reads as follows:
\begin{itemize}
\item
For every $\beta > 0$ there exist $C(\beta) < \infty$ and $\delta(\beta) > 0$ such that
\begin{equation} 
\label{eq:boundsfe}
0 \leq \f^\que (\beta, h_c^\que(\beta) - \delta) \leq
C(\beta) \delta^2 \qquad \forall\, 0 \leq \delta \leq \delta(\beta),
\end{equation}
i.e., the quenched phase transition is at least of second order. 
\end{itemize}
Unfortunately, the key assumption under which this smoothing inequality is derived is \emph{not obviously 
met} by spatially extended pinning: it does when the tail exponent of $\phi$ (recall \eqref{eq:phiscal}) 
satisfies $\theta  \in (2,\infty)$, but it is unclear whether it also does when $\theta \in (0,2]$. Indeed, the key 
assumption in \cite{cf:CadHo} requires that $Z_{N,\beta,h}^{\omega,\rc} \geq N^{-\gamma} c^\omega_{\beta,h}$ 
for some $\gamma>0$ with $\bbE[\log c^\omega_{\beta,h}]>-\infty$. Applying Jensen's inequality to \eqref{eq:Zc}, 
we get
\begin{equation}
Z_{N,\beta,h}^{\omega,\rc} \geq \ee^{\sum_{n=1}^N (\beta\omega_n - h) \e[\phi(S_n) \mid S_N=0]}\,
\p(S_N=0), \qquad N\in\N.
\end{equation}
In general, if $\p(S_N=0) \geq CN^{-\gamma}$ for some $\gamma,C>0$ and all $N\in\N$ and,
furthermore, $\sup_{N\in\N} \sum_{n=1}^N |\e[\phi(S_n) \mid S_N=0]| < \infty$, then the assumption is 
met (recall \eqref{eq:omass}). For the Bessel random walk, the former holds by \eqref{eq:Besseltail} 
for $\gamma=1-\alpha$, while the latter holds by \eqref{eq:lltub} when $\phi(x) \sim c \, |x|^{-\theta}$ 
with $\theta \in (2,\infty)$, because in that case, for $n \le N/2$ (by symmetry), $\sup_{N\in\N}\e[\phi(S_n) 
\mid S_N=0] \lesssim \phi(\sqrt{n}\,) \asymp n^{-\theta/2}$ is summable.

\medskip\noindent
{\bf 5.}  
Theorems~\ref{thm:weak1}--\ref{thm:weak3} give detailed information about the scaling of 
the annealed free energy and the annealed critical curve in the limit of weak interaction. 
The scaling limits correspond to annealed free energies and annealed critical curves for 
Brownian versions of the model involving the Bessel process $X^\alpha$, which are 
interesting in their own right. The result is only valid for the Bessel random walk, and shows 
a \emph{trichotomy} depending on the parameters $\alpha$ and $\theta$. 
\begin{itemize}
\item
The regime $\theta \in (0,1-\alpha)$ corresponds to a \emph{long-range} interaction potential 
and is \emph{not} pinning-like. When localized, the continuum polymer spends a positive fraction of the 
time near any height $x \in \R$, and this fraction tends to zero as $|x| \downarrow 0$ or $|x|
\to \infty$. Away from $0$ it does \emph{not} behave like the Bessel process conditioned to return 
to $0$.
\item
The regime $\theta \in (1-\alpha,2(1-\alpha))$ corresponds to an \emph{intermediate-range} 
interaction potential and exhibits some pinning-like features. When localized, the continuum 
polymer visits $0$ a positive fraction of the time. Away from $0$ it does \emph{not} behave like the 
Bessel process conditioned to return to $0$.    
\item
The regime $\theta \in (2(1-\alpha),\infty)$ corresponds to a \emph{short-range} interaction potential
and is pinning-like. When localized, the continuum polymer visits $0$ a positive fraction of the 
time. Away from $0$ it behaves like the Bessel process conditioned to return to $0$.  
\end{itemize}
In the last regime the behaviour is similar to that of the homogeneous pinning model with 
$\phi(x) = c\ind_{\{x=0\}}$, for which it is known that $h^\ann_c(\beta) \sim \tfrac12 c \beta^2$, 
$\beta \downarrow 0$ (see Giacomin~\cite{cf:Gi1}, \cite{cf:Gi2}, den Hollander~\cite{cf:dHo}). 
In fact, the proof of Theorem~\ref{thm:weak3} will show that the scaling in the last regime is 
valid for any $\phi$ such that $c^*[\phi^2]$ and $c^*[\phi]$ are finite, i.e., \eqref{eq:phiscal} may
be replaced by the weaker condition $\phi(x) = O(|x|^{-2(1-\alpha)-\epsilon})$ for some 
$\epsilon>0$.
  
\medskip\noindent
{\bf 6.}
The three regimes for $\theta$ represent three \emph{universality classes}. The critical cases 
$\theta = 1-\alpha$ and $\theta = 2(1-\alpha)$ are more delicate and we have skipped them. 
Also, we have not investigated what happens when the scaling of the interaction potential in 
\eqref{eq:phiscal} is modulated by a slowly varying function. For the same reason we have 
assumed that the error term in \eqref{eq:dasymp} is $O(|x|^{-(1+\epsilon)})$ with $\epsilon>0$ 
rather than $o(|x|^{-1})$, since the latter may give rise to modulation by slowly varying functions 
in \eqref{eq:Besseltail} and \eqref{eq:ctildeasymp} (see Alexander~\cite{cf:Al}).

\medskip\noindent
{\bf 7.}
Theorems~\ref{thm:weak1}--\ref{thm:weak3} are deduced from scaling properties of the annealed 
partition function. They are not specific to the annealed model. In fact, analogous results hold for every 
homogeneous pinning potential that is a linear combination of potentials with a polynomial tail. See 
Section~\ref{s:annealed} for more details.
  
\medskip\noindent
{\bf 8.}
By the considerations made in Section~\ref{s:MBG}, the annealed model defined in 
\eqref{eq:Zann}--\eqref{eq:psi} is localized (i.e., $\f^\ann(\beta,h) > 0$, $h < h_c^\ann(\beta)$) 
if and only if
\begin{equation}
\label{eq:condloc}
\sum_{m\in\N} \e \bigg[ \ee^{\sum_{n=1}^{m} \psi_{\beta,h}(S_n)} \, 
\ind_{\{\tau_1 = m\}} \bigg] > 1,
\end{equation}
where we recall that $\tau_1$ denotes the first return time of $S$ to $0$. Although the starting 
point $0$ seems to play a special role in \eqref{eq:condloc}, it can be shown that the criterion 
in \eqref{eq:condloc} is invariant under spatial shifts of $\psi_{\beta,h}$ (see Appendix~\ref{appC}). 

A natural question is what happens when the random walk $S$ is 
\emph{transient}, i.e., $\p(\tau_1 < \infty) = \sum_{n\in\N} K(n) =: r < 1$. For the constrained 
partition function $Z_{N,\beta,h}^{\omega,\rc}$, working with a transient renewal process with 
law $K$ is equivalent to working with a recurrent renewal process with law $K/r$ and adding a 
\emph{depinning term} $\sum_{n=1}^N (\log r) \ind_{\{S_n = 0\}}$ in the exponential in 
\eqref{eq:Zc}. This amounts to replacing $\psi_{\beta,h}(x)$ by $\psi_{\beta,h}(x) + 
(\log r)\ind_{\{x=0\}}$, and so instead of \eqref{eq:condloc} the localization condition 
for the annealed model becomes
\begin{equation} 
\label{eq:condlocalt}
\sum_{m\in\N} \e\left[\ee^{\sum_{n=1}^{m} \psi_{\beta,h}(S_n)} \, \ind_{\{\tau_1 = m\}} \right] 
> \frac{1}{r}.
\end{equation}
\begin{itemize}
\item[(I)] 
For the \emph{copolymer model} we have $\phi(x) = \phi^\cop(x) = \ind_{\{x \leq 0\}}$ 
(recall \eqref{eq:phipincop}) and $\psi_{\beta,h}(x) = (\tfrac12\beta^2-h) \ind_{(-\infty,0]}(x)$. 
Therefore $h_c^\ann(\beta) = \tfrac12\beta^2$ and, in fact, the left-hand side of \eqref{eq:condloc} 
is $\leq 1$ for $h \leq h_c^\ann(\beta)$ and is $=\infty$ for $h > h_c^\ann(\beta)$. This means 
that the annealed critical curve $h_c^\ann$ does not depend on $r$, and hence neither do the 
bounds in \eqref{eq:boundshc}. In other words, making the underlying renewal process transient 
or, equivalently, adding a homogeneous depinning term at zero, does not modify the annealed 
critical curve of the copolymer model. In essence this is due to the fact that the copolymer 
potential is \emph{long range} (i.e., $\psi_{\beta,h}(x)$ does not vanish as $x \to -\infty$).
\item[(II)] 
For the \emph{pinning model}, adding a depinning term at zero amounts to shifting $h$ and 
this may have an effect. In essence this is due to the fact that the pinning potential $\phi(x) 
= \phi^\pin(x) = \ind_{\{x=0\}}$ is \emph{short range}.
\end{itemize}

%%%

\subsection{Open problems and outline}
\label{ss:open}

\begin{itemize}
\item[1.]
Are the inequalities in \eqref{eq:boundshc} strict for all $\beta>0$? For the \emph{copolymer 
model} ($\phi(x)=1_{\{x \leq 0\}}$) the answer is yes for all $\alpha>0$ (Bolthausen, den 
Hollander and Opoku~\cite{cf:BodHoOp}). Moreover, it is known that (Bolthausen and den 
Hollander~\cite{cf:BodHo}, Caravenna and Giacomin~\cite{cf:CaGi}) 
\begin{equation}
\lim_{\beta \downarrow 0} \beta^{-1} h^\ann_c(\beta) = 1,
\qquad \lim_{\beta \downarrow 0} \beta^{-1} h^\que_c(\beta) = C(\alpha),
\end{equation} 
with $C(\alpha) \in (1/(1+\alpha),1)$ for $0<\alpha<1$ (Bolthausen, den Hollander and 
Opoku~\cite{cf:BodHoOp}) and $C(\alpha)=1/(1+\alpha)$ for $\alpha \geq 1$ (Berger, 
Caravenna, Poisat, Sun and Zygouras~\cite{cf:BeCaPoSuZy}). For the \emph{pinning model} 
($\phi(x)=1_{\{x=0\}}$) the answer depends on $\alpha$: no for $0<\alpha<\tfrac12$ 
(the upper bound is an equality), yes for $\alpha>\tfrac12$. Moreover, it is known that
\begin{equation}
\lim_{\beta \downarrow 0} \beta^{-2} h_c^\ann(\beta) = \tfrac12,
\qquad \lim_{\beta \downarrow 0} \beta^{-2} h_c^\que(\beta) = C'(\alpha),
\end{equation} 
with $C'(\alpha)=\tfrac12$ for $0<\alpha<1$ and $C'(\alpha) \in (0,\tfrac12)$ for $\alpha>1$.
As a matter of fact, refined estimates are available: 
\begin{itemize}
\item
$\alpha \in (0,\tfrac12)$: $h^\ann_c(\beta) = h^\que_c(\beta)$ for $\beta > 0$ small enough 
(Alexander~\cite{cf:Al0}, see also Cheliotis and den Hollander~\cite{cf:ChdHo});
\item
$\alpha \in (\tfrac12,1)$: $h^\ann_c(\beta)-h^\que_c(\beta) \sim c_\alpha \, \beta^{2\alpha/(2\alpha-1)}$ 
as $\beta \downarrow 0$ for a universal constant $c_\alpha \in (0,\infty)$ (Caravenna, Toninelli 
and Torri~\cite{cf:CaToTo}, and previously Alexander and Zygouras~\cite{cf:AlZy},
Derrida, Giacomin, Lacoin and Toninelli~\cite{cf:DeGiLaTo}).
\item
$\alpha = \frac{1}{2}$: $h^\ann_c(\beta)-h^\que_c(\beta) = \exp(-\frac{\pi}{2\beta^2}[1+o(1)])$
(Berger and Lacoin~\cite{cf:BeLa}, and previously Giacomin, Lacoin and Toninelli~\cite{cf:GiLaTo}).
\end{itemize}
For an overview,  we refer the reader to Giacomin~\cite{cf:Gi2}.
\item[2.]
Determine the order of the quenched phase transition. For the copolymer model it is 
known that the phase transition is of infinite order when $\alpha=0$ (Berger, Giacomin and 
Lacoin~\cite{cf:BeGiLa}). The same is conjectured to be true for $\alpha \in (0,1)$.
\item[3.]
Identify the scaling for weak interaction of the annealed model in the critical cases $\theta 
= 1-\alpha$ and $\theta = 2(1-\alpha)$.
\item[4.]
Identify the scaling for weak interaction of the quenched model. Because of Theorem~\ref{thm:boundshc}, 
the \emph{same} exponent $E$ as in \eqref{eq:exponent} applies.
\item[5.]
The qualitative shape of the critical curve in Fig.~\ref{fig-critcurveque} depends on our assumption 
in \eqref{eq:phidecay} that $\phi \geq 0$. A reflected picture holds when $\phi \leq 0$. It appears 
that for $\phi$ with mixed signs there are two critical curves $\beta \mapsto h_ {c,1}^\que(\beta)$ 
and $\beta \mapsto h_{c,2}^\que(\beta)$, separating a single quenched delocalized phase 
$\cD^\que$ from two quenched localized phases $\cL^\que_1$ and $\cL^\que_2$ 
that lie above $\cD^\que$, respectively, below $\cD^\que$. What are the properties 
of these critical curves? 
\item[6.]
What happens when $\beta=\beta_N$ and $h=h_N$ with $\beta_N,h_N \downarrow 0$ as $N\to\infty$. 
\item[7.]
Is it possible to include non-nearest-neighbour random walks?  
\end{itemize}

\bigskip\noindent
{\bf Outline.} The remainder of this paper is organized as follows. Theorem~\ref{thm:varrep} is 
proved in Section~\ref{s:varcrit}, Theorem~\ref{thm:boundshc} in Section~\ref{s:MBG} and 
Theorems~\ref{thm:weak1}--\ref{thm:weak3} in Section~\ref{s:annealed}. Appendices~\ref{appA}
and~\ref{appB} collect a few technical facts that are needed along the way.

%%%%%%%%%%%%%%%%%%%%%%%%%%%%%%%%%%%

\section{Proof of Theorem~\ref{thm:varrep}}
\label{s:varcrit}

In Section~\ref{ss:LDPs} we formulate annealed and quenched large deviation principles (LDPs) 
that are an adaptation to our model of the LDPs developed in Birkner~\cite{cf:Bi} and Birkner, 
Greven and den Hollander~\cite{cf:BiGrdHo}. The latter concern LDPs for \emph{random sequences 
of words} cut out from \emph{random sequences of letters} according to a renewal process. In 
Section~\ref{ss:varcritmeth} we formulate variational characterizations of the annealed and quenched 
critical curves that are an adaptation of the characterizations derived for the pinning model in 
Cheliotis and den Hollander~\cite{cf:ChdHo} and for the copolymer model in Bolthausen, den 
Hollander and Opoku~\cite{cf:BodHoOp}. In Section~\ref{ss:Varadhan} we explain how the 
variational characterizations follow from the LDPs via Varadhan's lemma.

%%%%%%%%%%%%%%

\subsection{Annealed and quenched LDP}
\label{ss:LDPs}

Our starting observation is that the partition function in \eqref{eq:model1} 
depends on the \emph{sequence of words} $Y=(Y_i)_{i\in\N}$ determined 
by the disorder and by the excursions of the polymer, namely,
\begin{equation}
Y_i  = Y_i(\omega,S) :=\, \big((\omega_{\tau_{i-1}+1}, \ldots, \omega_{\tau_i}),
(S_{\tau_{i-1}+1}, \ldots, S_{\tau_i})\big), \qquad i \in \N,
\end{equation}
where $\tau=(\tau_i)_{i\in\N}$ is the \emph{sequence of epochs} of the successive 
visits of the polymer to zero ($\tau_0=0$). Note that the random variables $Y_i$ 
take their values in the space $\tilde{\Gamma} := \bigcup_{n\in\N} \Gamma^n$ 
with $\Gamma := \R \times \Z$. 

To capture the role of $Y$, we introduce its empirical process,
\begin{equation}
R_M = R_M^\omega := \frac{1}{M} \sum_{i=0}^{M-1} 
\delta_{\tilde\theta^i (Y_1, \ldots, Y_M)^\per} \in \cP^\inv(\tilde\Gamma^\N),
\end{equation}
where $\cP^\inv(\tilde\Gamma^\N)$ denotes the set of probability measures on 
$\tilde\Gamma^\N$ that are invariant under the left-shift $\tilde\theta$ acting on 
$\tilde\Gamma^\N$. The superscript $\omega$ reminds us that the random variables 
$Y_i$ are functions of $\omega$. We must average over $S$ while keeping $\omega$ 
fixed. Note that, under the \emph{annealed law} $\bbP \otimes \p$, $Y$ is i.i.d.\ with 
the following marginal law $q_0$ on $\tilde\Gamma$:
\begin{equation} 
\label{eq:q0}
\begin{split}
&q_0\big( (\dd x_1, \ldots, \dd x_n) \times \{(s_1, \ldots, s_n)\} \big) \\
&\qquad := (\bbP \otimes \p)\big(Y_1 \in (\dd x_1, \ldots, \dd x_n) \times \{(s_1, \ldots, s_n)\}\big) \\
&\qquad = K(n) \, \nu(\dd x_1) \cdots \nu(\dd x_n) \,
\p\big( (S_1, \ldots, S_n) = (s_1, \ldots, s_n) \,\big|\, \tau_1 = n \big),\\
&n\in\N,\,x_1,\ldots,x_n \in \R,\,s_1,\ldots,s_n \in \Z, 
\end{split}
\end{equation}
where $K(n) := \p(\tau_1 = n)$ and $\nu(\dd x) := \bbP(\omega_1 \in \dd x)$. 

The \emph{specific relative entropy of $Q$ w.r.t.\ $q_0^{\otimes\N}$} is defined by
\begin{equation}
\label{eq:spentrdef}
H(Q \mid q_0^{\otimes\N}) := \lim_{N\to\infty} \frac{1}{N}\, h(\tilde\pi_N Q \mid q_0^N),
\end{equation}
where $\tilde\pi_N Q \in \cP(\tilde\Gamma^N)$ denotes the projection of $Q$ onto the first 
$N$ words, $h(\,\cdot\mid\cdot\,)$ denotes relative entropy, and the limit is non-decreasing. 
The following annealed LDP is standard (see Dembo and Zeitouni~\cite[Section 6.5]{cf:DeZe}).

\begin{proposition}
\label{aLDP}
{\rm {\bf [Annealed LDP]}}
The family $(\bbP \otimes \p)(R_M^\omega \in \cdot\,)$, $M\in\N$, satisfies the LDP on 
$\cP^{\mathrm{inv}}(\tilde\Gamma^\N)$ with rate $M$ and with rate function $I^{\mathrm{ann}}$ 
given by
\begin{equation}
\label{eq:Ianndef}
I^\ann(Q) := H\big(Q \mid q_0^{\otimes\N}\big),
\qquad Q \in \cP^{\mathrm{inv}}(\tilde\Gamma^\N).
\end{equation}
This rate function is lower semi-continuous, has compact level sets, has a unique
zero at $q_0^{\otimes\N}$, and is affine.
\end{proposition}

The quenched LDP is more delicate and requires extra notation. The reverse operation of 
\emph{cutting} words out of a sequence of letters is \emph{glueing} words together into a 
sequence of letters. Formally, this is done by defining a \emph{concatenation map} $\kappa$ 
from $\tilde\Gamma^\N$ to $\Gamma^\N$. This map induces in a natural way a map from 
$\cP(\tilde\Gamma^\N)$ to $\cP(\Gamma^\N)$, the sets of probability measures on 
$\tilde\Gamma^\N$ and $\Gamma^\N$ (endowed with the topology of weak convergence). 
The concatenation $q_0^{\otimes\N} \circ \kappa^{-1}$ of $q_0^{\otimes\N}$ equals 
$\nu^{\otimes\N}$, as is evident from (\ref{eq:q0}).

For $Q\in\cP^\inv(\tilde\Gamma^\N)$, let $m_Q:=E_Q(\tau_1) \in [1,\infty]$ be the average 
word length under $Q$ ($E_Q$ denotes expectation under the law $Q$ and $\tau_1$
is the length of the first word). Let
\begin{equation}
\label{eq:Pfin}
\cP^{\mathrm{inv,fin}}(\tilde\Gamma^\N) := \{Q\in\cP^{\mathrm{inv}}(\tilde\Gamma^\N)
\colon\,m_Q<\infty\}.
\end{equation}
For $Q\in\cP^{\mathrm{inv,fin}}(\tilde\Gamma^\N)$, define
\begin{equation}
\label{eq:PsiQdef}
\Psi_Q := \frac{1}{m_Q} E_Q\left[\sum_{k=0}^{\tau_1-1}
\delta_{\theta^k\kappa(Y)}\right] \in \cP^{\mathrm{inv}}(\Gamma^\N).
\end{equation}
Think of $\Psi_Q$ as the shift-invariant version of $Q\circ\kappa^{-1}$ obtained
after \emph{randomizing} the location of the origin. This randomization is necessary
because a shift-invariant $Q$ in general does not give rise to a shift-invariant
$Q\circ\kappa^{-1}$.

The following quenched LDP is a straight adaptation of the one derived in Birkner, Greven and 
den Hollander~\cite{cf:BiGrdHo}.

\begin{proposition}
\label{qLDP}
{\rm {\bf [Quenched LDP]}}
For $\bbP$-a.e.\ $\omega$ the family $\p(R_M^\omega \in \cdot\,)$, $M\in\N$, satisfies 
the LDP on $\cP^\inv(\tilde\Gamma^\N)$ with rate $M$ and with rate function given by
\begin{equation}
\label{eq:Iqueiden}
I^\que(Q) := \left\{\begin{array}{ll}
H(Q \mid q_0^{\otimes\N}) + \alpha\, m_Q\, H(\Psi_Q \mid \nu^{\otimes\N}), 
&Q \in \cP^{\inv,\fin}(\tilde\Gamma^\N),\\
\lim_{\delta\downarrow 0} \inf_{Q' \in B_\delta(Q) \cap \cP^{\inv,\fin}(\tilde\Gamma^\N)} I^\que(Q'),
&Q \in \cP^{\inv}(\tilde\Gamma^\N) \setminus \cP^{\inv,\fin}(\tilde\Gamma^\N), 
\end{array}
\right.
\end{equation}
where $\alpha$ is the exponent in \eqref{eq:astau} and $B_\delta(Q)$ is the $\delta$-ball around 
$Q$ (in any appropriate metric). This rate function is lower semi-continuous, has compact level sets, 
has a unique zero at $q_0^{\otimes\N}$, 
and is affine.
\end{proposition}

\begin{remark}
{\rm In \cite{cf:BiGrdHo} a formula was claimed for $I^\que$ on $\cP^{\inv}(\tilde\Gamma^\N) 
\setminus \cP^{\inv,\fin}(\tilde\Gamma^\N)$ based on a truncation approximation for the average 
word length. As pointed out by Jean-Christophe Mourrat (private communication), the proof of 
this formula in \cite{cf:BiGrdHo} is flawed. The formula itself may still be correct, but no proof is 
currently available. In the present paper we will only need to know $I^\que$ on $\cP^{\inv,\fin}
(\tilde\Gamma^\N)$.}
\end{remark}

%%%%%%%%%%%%%%%%%%%%%%%%%%%%%

\subsection{Variational criterion for localization}
\label{ss:varcritmeth}

For $N\in\N$, let $\ell_N$ be the number of returns to zero of the polymer before epoch $N$, i.e.,
\begin{equation}
\ell_N := \max\{i \in \N_0\colon\, \tau_i \leq N\}.
\end{equation}
Let $\Phi\colon\,\tilde\Gamma \to \R$ be defined by
\begin{equation} 
\label{eq:Phi}
\Phi_{\beta,h}\big((x_1, \ldots, x_n), (s_1, \ldots, s_n)\big) 
:= \sum_{m=1}^n (\beta x_m - h) \phi(s_m).
\end{equation}
Then the constrained quenched partition function defined in \eqref{eq:Zc} can be written 
as
\begin{equation}
\label{eq:Zquerepr}
Z_{N, \beta, h}^{\omega,\rc} 
= \e\left[ \ee^{\sum_{i=1}^{\ell_N} \Phi_{\beta,h}(Y_i)} \,\ind_{\{S_N = 0\}} \right]
= \e\left[ \ee^{\ell_N \int_{\tilde\Gamma} \Phi_{\beta,h}\, \dd(\tilde\pi_1\R_{\ell_N}^\omega)} 
\,\ind_{\{S_N = 0\}} \right],
\end{equation}
while the constrained annealed partition function defined in \eqref{eq:Zcann} can be 
written as
\begin{equation}
\label{eq:Zannrepr}
Z_{N, \beta, h}^{\ann,\rc} 
= (\bbE \otimes \e)\left[ \ee^{\sum_{i=1}^{\ell_N} \Phi_{\beta,h}(Y_i)} \,\ind_{\{S_N = 0\}} \right]
= (\bbE \otimes \e)\left[ \ee^{\ell_N \int_{\tilde\Gamma} \Phi_{\beta,h}\, \dd(\tilde\pi_1\R_{\ell_N}^\omega)} 
\,\ind_{\{S_N = 0\}} \right],
\end{equation}
where $\tilde\pi_1 Q$ denotes the projection of $Q$ onto the space $\tilde\Gamma$ 
of the first word. 

With the help of Propositions~\ref{aLDP}--\ref{qLDP} we can derive the following 
\emph{variational characterization} of the annealed and the quenched critical curve.
Note that $\f^\ann(\beta,h) > 0$ if and only if $h<h_c^\ann(\beta)$ and $\f^\que(\beta,h) 
> 0$ if and only if $h<h_c^\que(\beta)$.

\begin{theorem}
\label{thm:annchar}
{\rm {\bf [Annealed localization]}}
For every $\beta,h>0$,
\begin{equation}
\label{eq:varcritann}
\f^\ann(\beta,h) > 0 \quad \iff \quad
\sup_{ {Q \in \cP^\inv(\tilde\Gamma^\N)\colon} \atop {m_Q < \infty, \,I^\ann(Q) < \infty} }
\bigg\{ \int_{\tilde\Gamma} \Phi_{\beta,h}\, \dd (\tilde\pi_1 Q) - I^\ann(Q) \bigg\} > 0.
\end{equation}
\end{theorem}

\begin{theorem}
\label{thm:quechar}
{\rm {\bf [Quenched localization]}}
For every $\beta,h>0$,
\begin{equation}
\label{eq:varcrit}
\f^\que(\beta,h) > 0 \quad \iff \quad
\sup_{ {Q \in \cP^\inv(\tilde\Gamma^\N)\colon} \atop {m_Q < \infty, \, I^\ann(Q) < \infty} }
\bigg\{ \int_{\tilde\Gamma} \Phi_{\beta,h}\, \dd (\tilde\pi_1 Q) - I^\que(Q) \bigg\} > 0.
\end{equation}
\end{theorem}

\noindent
As we will see below, the role of the conditions $m_Q<\infty$ and $I^\ann(Q)< \infty$ under the two
suprema is to ensure that $\int_{\tilde\Gamma} \Phi_{\beta,h}\,\dd(\tilde\pi_1 Q) < \infty$, so that the suprema 
are \emph{well defined}. The condition $m_Q<\infty$ under the second supremum allows us to use the 
\emph{representation} in \eqref{eq:Iqueiden}. We will see in Section~\ref{s:MBG} how the variational 
formulas in \eqref{eq:varcritann}--\eqref{eq:varcrit} can be exploited.

%%%%%%%%%%%%%%%%%%%%%%%%%%%%

\subsection{Proof of Theorems~\ref{thm:annchar}--\ref{thm:quechar}}
\label{ss:Varadhan}

The proof uses arguments developed in Bolthausen, den Hollander and Opoku~\cite{cf:BodHoOp}. 
Theorems~\ref{thm:annchar}--\ref{thm:quechar} follow from Propositions~\ref{aLDP}--\ref{qLDP} 
with the help of Varadhan's lemma applied to \eqref{eq:Zquerepr}--\eqref{eq:Zannrepr}. The only 
difficulty we need to deal with is the fact that both $Q \mapsto m_Q$ and $Q \mapsto \Phi^*_{\beta,h}(Q) 
= \int_{\tilde\Gamma} \Phi_{\beta,h}\, \dd(\tilde\pi_1Q)$ are neither bounded nor continuous in the weak 
topology. Therefore an \emph{approximation argument} is required, which is worked out in detail
in \cite[Appendix A--D]{cf:BodHoOp} for the case of the copolymer interaction potential in \eqref{eq:phipincop}. 
This approximation argument shows why the restriction to $m_Q<\infty$ and $I^\ann(Q)<\infty$ may 
be imposed, a key ingredient being that $I^\ann(Q)<\infty$ implies $\Phi^*_{\beta,h}(Q)<\infty$. The 
proof in \cite[Appendix A--D]{cf:BodHoOp} readily carries over because our condition on the interaction 
potential in \eqref{eq:phidecay} reflects the properties of the copolymer interaction potential. We 
sketch the main line of thought. Throughout the sequel $\beta,h>0$ are fixed. 

\begin{proof}[Proof of Theorem~\ref{thm:quechar}]
Following the argument in \cite[Appendix A]{cf:BodHoOp}, we show that
\begin{itemize}
\item[(1)]
For every $g>0$, $M \mapsto \Psi_{\beta,h}(R_M^\omega)-gm_{R_M^\omega}$ is bounded 
$\omega$-a.s.  
\item[(2)]
For every $\rho\in\cP(\N)$, $\nu\in\cP(\R)$ and $p=(p_n)_{n\in\N}$ with $p_n \in \cP(\Z^n)$, there 
exist $\gamma>0$ and $K=K(\rho,\nu,p;\gamma)>0$ such that $\Phi^*_{\beta,h}(Q) \leq \gamma 
h(\tilde{\pi}_1Q \mid q_{\rho,\nu,p}) + K$ for all $Q\in\cP^{\mathrm{inv}}(\tilde\Gamma^\N)$ with 
$h(\tilde{\pi}_1Q \mid q_{\rho,\nu,p})<\infty$, where (compare with \eqref{eq:q0})
\begin{equation}
\begin{aligned}
&q_{\rho,\nu,p}\big( (\dd x_1, \ldots, \dd x_n) \times \{(s_1, \ldots, s_n)\} \big) 
= \rho(n) \, \nu(\dd x_1) \cdots \nu(\dd x_n) \, p_n(s_1,\ldots,s_n),\\
&n\in\N,\,x_1,\ldots,x_n \in \R,\,s_1,\ldots,s_n \in \Z. 
\end{aligned}
\end{equation}
\end{itemize}
The proof uses the fact that the conditions in \eqref{eq:phidecay} allow us to approximate $\phi$ by 
a multiple of $\phi^{\mathrm{cop}}$ (recall \eqref{eq:phipincop}) \emph{uniformly} on $\Z \setminus 
[-L,L]$ at arbitrary precision as $L\to\infty$. The proof also uses a concentration of measure estimate 
for the disorder, which is proved in  \cite[Appendix D]{cf:BodHoOp}.
 
For $g>0$, define the quenched free energy 
\begin{equation}
\label{eq:form1}
\f^\que(\beta,h;g) := \lim_{N\to\infty} \frac{1}{N} \log Z_{N, \beta, h, g}^\que, 
\end{equation}
where 
\begin{equation}
\label{eq:form2}
Z_{N, \beta, h, g}^\que := \e\left[\ee^{N\big\{\Phi^*_{\beta,h}(R_M^\omega)
- gm_{R_M^\omega}\big\}}\right]
\end{equation}
is the quenched partition function in which every letter gets an energetic penalty $-g$. Following 
the argument in \cite[Appendix B]{cf:BodHoOp}, we use (1) and (2) to show that, for every $g>0$,
\begin{equation}
\label{eq:glim}
\f^\que(\beta,h;g) = \sup_{ {Q \in \cR\colon}
\atop {m_Q < \infty, \, I^\ann(Q) < \infty} }
\bigg\{ \int_{\tilde\Gamma^\N} \Phi^*_{\beta,h}(Q) - g m_Q - I^\ann(Q) \bigg\},
\end{equation}
where $\cR$ is the set of shift-invariant probability measures under which the concatenation of 
words produces a letter sequence that has the \emph{same asymptotic statistics as a typical 
realisation of} $Y$, i.e., 
\begin{equation}
\cR := \left\{Q \in \cP^\inv(\tilde\Gamma^\N)\colon\, w-\lim_{M\to\infty} \frac{1}{M} \sum_{k=0}^{M-1}
\delta_{\theta^k\kappa(Y)} = \mu^{\otimes\N}\,\, Q-\text{a.s.}\right\} 
\end{equation}
($w-\lim$ means weak limit). The proof of \eqref{eq:glim} carries over verbatim. What \eqref{eq:glim}
says is that Varadhan's lemma applies to \eqref{eq:form1}--\eqref{eq:form2} because of the 
control enforced by (1) and (2). 

Following the argument in \cite[Appendix C]{cf:BodHoOp}, we show that
\begin{equation}
\lim_{g \downarrow 0} \f^\que(\beta,h;g) = \f^\que(\beta,h)
\end{equation}
with
\begin{equation}
\label{eq:glimalt}
\f^\que(\beta,h) = \sup_{ {Q \in \cP^\inv(\tilde\Gamma^\N)\colon}
\atop {m_Q < \infty, \, I^\ann(Q) < \infty} }
\bigg\{ \int_{\tilde\Gamma^\N} \Phi^*_{\beta,h}(Q) - I^\que(Q) \bigg\}.
\end{equation}
Here, in the passage from \eqref{eq:glim} to \eqref{eq:glimalt}, the constraint in $\cR$ disappears from 
the variational characterization, while $I^\ann(Q)$ is replaced by $I^\que(Q)$. The reason is that for 
every $Q \in \cP^\inv(\tilde\Gamma^\N)$ there exists a sequence $(Q_n)_{n\in\N}$ in $\cR$ such that 
$Q = w-\lim_{n\to\infty} Q_n = Q$ and $\lim_{n\to\infty} I^\ann(Q_n) = I^\que(Q)$. The proof of 
\eqref{eq:glimalt} carries over verbatim. The supremum in the right-hand side of \eqref{eq:glimalt} is 
the same as the supremum in the right-hand side of \eqref{eq:varcrit}. 
\end{proof}

\begin{proof}[Proof of Theorem~\ref{thm:annchar}]
Varadhan's lemma also applies to
\begin{equation}
\label{eq:form2alt}
Z_{N, \beta, h}^\ann := \e\left[\ee^{N\Phi^*_{\beta,h}(R_M^\omega)}\right],
\end{equation} 
and yields
\begin{equation}
\label{eq:form1alt}
\f^\ann(\beta,h) := \lim_{N\to\infty} \frac{1}{N} \log Z_{N, \beta, h}^\ann 
\end{equation}
with 
\begin{equation}
\label{eq:glimaltalt}
\f^\ann(\beta,h) = \sup_{ {Q \in \cP^\inv(\tilde\Gamma^\N)\colon}
\atop {m_Q < \infty, \, I^\ann(Q) < \infty} }
\bigg\{ \int_{\tilde\Gamma^\N} \Phi^*_{\beta,h}(Q) - I^\ann(Q) \bigg\}.
\end{equation}
Again, this works because of the control enforced by (2). The supremum in the right-hand side of 
\eqref{eq:glimaltalt} is the same as the supremum in the right-hand side of \eqref{eq:varcritann}. 
\end{proof}

%%%%%%%%%%%%%%%%%%%%%%%%%%

\section{Proof of Theorem~\ref{thm:boundshc}}
\label{s:MBG}

The upper bound in \eqref{eq:boundshc} is immediate from Theorems~\ref{thm:annchar}--\ref{thm:quechar}  
and the inequality $I^\que \geq I^\ann$ (see also \eqref{eq:annbound}). In Sections~\ref{ss:varcrit}--\ref{ss:applvarcrit} 
we prove the lower bound in \eqref{eq:boundshc}. This lower bound is the analogue of what for the copolymer 
model is called the Monthus-Bodineau-Giacomin lower bound (see Giacomin~\cite{cf:Gi1}, den 
Hollander~\cite{cf:dHo}).

%%%%%%%%%%%%%%

\subsection{A sufficient criterion for quenched localization}
\label{ss:varcrit}

The quenched rate function can be written as
\begin{equation}
I^\que(Q) = (1+\alpha) I^\ann(Q) - \alpha\, R(Q), \qquad  m_Q < \infty,
\end{equation}
with
\begin{equation}
R(Q) := H(Q \mid q_0^{\otimes\N}) - m_Q\, H(\Psi_Q \mid \nu^{\otimes\N}).
\end{equation}
It can be shown that $R(Q) \geq 0$ for all $Q\in \cP^\inv(\tilde\Gamma^\N)$: $R(Q)$ has 
the meaning of a \emph{concatenation entropy} (see Birkner, Greven and den 
Hollander~\cite{cf:BiGrdHo}). Therefore, dropping this term in \eqref{eq:varcrit} we obtain 
the following \emph{sufficient} criterion for quenched localization: 
\begin{equation}
\label{eq:varcritsimple}
\f^\que(\beta,h) > 0 \quad \Longleftarrow 
\sup_{ {Q \in \cP^\inv(\tilde\Gamma^\N)\colon} \atop {m_Q < \infty, \, I^\ann(Q) < \infty} }
\bigg\{ \int_{\tilde\Gamma} \Phi\, \dd (\tilde\pi_1 Q) - (1+\alpha)I^\ann(Q) \bigg\} > 0.
\end{equation}
The right-hand side resembles the \emph{necessary and sufficient} criterion for annealed localization 
in \eqref{eq:varcritann}, the only difference being the extra factor $1+\alpha$.

%%%%%%%%%

\subsection{Reduction}
\label{ss:solvarcrit}

Among the laws $Q \in \cP^\inv(\tilde\Gamma^\N)$ with a given marginal law $q \in \cP(\tilde\Gamma)$, 
the product law $Q = q^{\otimes\N}$ is the unique minimizer of the specific relative entropy $H(Q \mid
q_0^{\otimes\N})$. Therefore the right-hand side of \eqref{eq:varcritsimple} reduces to
\begin{equation}
\label{eq:varcritsimple2}
\sup_{ {q \in \cP^\inv(\tilde\Gamma)\colon} \atop  {m_q < \infty,\,h(q \mid q_0) < \infty} }
\bigg\{ \int_{\tilde\Gamma} \Phi\, \dd q - (1+\alpha) h(q \mid q_0) \bigg\} > 0,
\end{equation}
where $h(\cdot \mid \cdot)$ denotes relative entropy.

We next show that \eqref{eq:varcritsimple} reduces to an even simpler criterion. To that end, 
let $C_N := \bigcup_{n=1}^N \Gamma^n$ be the subset of words of length at most $N$. Consider the 
law $\hat q_N \in \cP(\tilde\Gamma)$ defined by
\begin{equation}
\frac{\dd\hat q_N}{\dd q_0} := \frac{\ee^{\frac{1}{1+\alpha}\Phi}\, \ind_{C_N}}{\cN_{\alpha,N}},
\end{equation}
where
\begin{equation} 
\label{eq:finN}
\cN_{\alpha,N} := \int_{\tilde\Gamma} \ee^{\frac{1}{1+\alpha}\Phi} \, \ind_{C_N} \, \dd q_0 < \infty
\end{equation}
is the normalizing constant. The latter is finite because, by \eqref{eq:Phi}, $\Phi$ restricted to $C_N$ is 
the sum of at most $N$ random variables with finite exponential moments. Note that also
\begin{equation} 
\label{eq:finent}
\int_{\tilde\Gamma} \Phi \, \ee^{\frac{1}{1+\alpha}\Phi} \, \ind_{C_N} \, \dd q_0 < \infty,
\end{equation}
which yields $h(\hat q_N \mid q_0) < \infty$. Trivially, $m_{\hat q_N} \leq N < \infty$. Therefore 
we are allowed to pick $q = \hat q_N$ in \eqref{eq:varcritsimple2}, so that \eqref{eq:varcritsimple2} 
is satisfied when
\begin{equation}
(1+\alpha) \log \cN_{\alpha,N} > 0.
\end{equation}
Since $N$ is arbitrary, this in turn is satisfied when 
\begin{equation}
\label{eq:varcritsimple3}
\cN_\alpha > 1 \quad \text{ with } \quad \cN_\alpha := \sup_{N\in\N} \cN_{\alpha,N} 
= \int_{\tilde\Gamma} \ee^{\frac{1}{1+\alpha}\Phi} \, \dd q_0,
\end{equation}
where $\cN_\alpha = \infty$ is allowed. Conversely, if $\cN_\alpha \leq 1$, then \eqref{eq:varcritsimple2} 
is not satisfied. Indeed, as soon as $\cN_\alpha < \infty$ we may introduce the law $\hat q \in 
\cP(\tilde\Gamma)$ defined by
\begin{equation}
\frac{\dd\hat q}{\dd q_0} := \frac{\ee^{\frac{1}{1+\alpha}\Phi}}{\cN_\alpha},
\end{equation}
and rewrite $h(q \mid q_0) = h(q \mid \hat q) - \frac{1}{1+\alpha} \int_{\tilde\Gamma} \Phi \, \dd q 
+ \log \cN_\alpha$, so that 
\begin{equation}
\int_{\tilde\Gamma} \Phi\, \dd q - (1+\alpha)h(q \mid q_0) 
= (1+\alpha) \log \cN_\alpha - h(q \mid \hat q) \leq 0,
\end{equation}
where the last inequality holds for any $q$ because $\cN_\alpha \leq 1$, and so \eqref{eq:varcritsimple2} 
fails. Thus, \eqref{eq:varcritsimple} reduces to 
\begin{equation}
\label{eq:simplloc}
\f^\que(\beta,h) > 0 \qquad \Longleftarrow \qquad \cN_\alpha > 1.
\end{equation}

%%%%

\subsection{Application}
\label{ss:applvarcrit}

As we remarked below \eqref{eq:varcritsimple}, \eqref{eq:varcritann} resembles \eqref{eq:varcritsimple}, 
the only difference being the factor $1+\alpha$ instead of $1$ in front of $I^\ann(Q) = H(Q \mid q_0^{\otimes\N})$. 
Therefore, repeating the above steps and recalling \eqref{eq:varcrit}, we conclude that
\begin{equation}
\label{eq:simpllocann}
\f^\ann(\beta,h) > 0 \qquad \iff \qquad \cN_0 > 1.
\end{equation}
It follows from \eqref{eq:Phi}, \eqref{eq:Zannrepr} and \eqref{eq:varcritsimple3} that the condition 
$\cN_\alpha > 1$ is equivalent to
\begin{equation}
\f^\ann\big(\tfrac{1}{1+\alpha} \beta, \tfrac{1}{1+\alpha} h\big) > 0,
\end{equation} 
i.e., $\frac{1}{1+\alpha}h < h_c^\ann(\frac{1}{1+\alpha}\beta)$, which by \eqref{eq:simplloc} implies 
$\f^\que(\beta,h) > 0$, i.e., $h < h_c^\que(\beta)$. This completes the proof of the lower bound 
in \eqref{eq:boundshc}. 

Recalling \eqref{eq:q0}, \eqref{eq:Phi} and the function $\psi_{\beta,h}$ defined in \eqref{eq:psi},
we may write $\cN_0$ as
\begin{equation} 
\label{eq:cN}
\begin{split}
\cN_0 = \int_{\tilde\Gamma} \ee^{\Phi} \, \dd q_0
&= \sum_{m\in\N} \bbE \bigg[ \e \bigg[ \ee^{\sum_{n=1}^{m} (\beta \omega_n - h) \phi(S_n)} 
\, \ind_{\{\tau_1 = m\}} \bigg] \bigg] \\
&= \sum_{m\in\N} \e \bigg[ \ee^{\sum_{n=1}^{m} \psi_{\beta,h}(S_n)} 
\, \ind_{\{\tau_1 = m\}} \bigg].
\end{split}
\end{equation}
We will analyse this expression in Section~\ref{s:annealed} in the weak interaction limit 
$\beta,h\downarrow 0$ for the Bessel random walk. 

Note that for $h=0$ the expression in \eqref{eq:cN} reduces to
\begin{equation}
\sum_{m\in\N} \bbE \bigg[ \e \bigg[ \ee^{\sum_{n=1}^{m} \beta\omega_n\phi(S_n)} 
\, \ind_{\{\tau_1 = m\}} \bigg] \bigg].
\end{equation}
By Jensen and the fact that $\phi\not\equiv 0$, this sum is $>1$ for all $\beta>0$. Hence
$F^\ann(\beta,0)>0$ for all $\beta>0$, which settles the claim made at the end of 
Section~\ref{ss:thmsgen}.

%%%%%%%%%%%%%%%%%%%%%%%%%

\section{Proof of Theorems~\ref{thm:weak1}--\ref{thm:weak3}}
\label{s:annealed}

To prove our scaling results for weak interaction, we will exploit the invariance principle in \eqref{eq:invprin}.
Recall \eqref{eq:Zann}--\eqref{eq:psi}. Since $\|\psi_{\beta,h}\|_\infty$ tends to zero as $\beta,h \downarrow 0$, 
we can do a \emph{weak coupling expansion} in the spirit of Caravenna, Sun and Zygouras~\cite{cf:CaSuZy}.

The proof is long and technical. In Section~\ref{sec:strategy} we outline the general strategy, which leads to three
tasks. These tasks are carried out in Sections~\ref{sec:hatCexpr}--\ref{sec:convCk}, respectively.

%%%%%%%%%%%%%%%%%%%%%%%

\subsection{General strategy}
\label{sec:strategy}

Fix $\hat\beta, \hat h \in (0,\infty)$ and $\beta_N, h_N$ such that, as $N\to\infty$,
\begin{equation}
\label{eq:betaN}
\beta_N \sim \hat \beta \times
\begin{cases}
N^{-(1-\theta)/2}, &\theta \in (0,1-\alpha), \\
N^{-\alpha/2}, &\theta \in (1-\alpha, 2(1-\alpha)), \\
N^{-\alpha/2}, &\theta \in (2(1-\alpha),\infty), 
\end{cases}
\end{equation}
and 
\begin{equation}
\label{eq:hN}
h_N \sim \hat h  \times 
\begin{cases}
N^{-(2-\theta)/2}, &\theta \in (0,1-\alpha), \\
N^{-(2-\theta)/2}, &\theta \in (1-\alpha, 2(1-\alpha)), \\
N^{-\alpha}, &\theta \in (2(1-\alpha),\infty). 
\end{cases}
\end{equation}
We will prove that, for any $T \in (0,\infty)$,
\begin{equation}
\label{eq:scalpart0}
\lim_{N\to\infty} Z^\ann_{TN, \beta_N, h_N} 
= \hat{Z}^\ann_{T,\hat{\beta},\hat{h}}  :=
\hat{\e}\left[\exp\left( \int_0^T \hat\psi_{\hat{\beta},\hat{h}}(X_t) \,  \dd t\right)\right] 
\end{equation}
(for ease of notation we pretend that $TN$ is integer), where we set
\begin{equation}
\label{eq:psi0def}
\hat\psi_{\hat{\beta},\hat{h}}(x) :=
\begin{cases}
\tfrac12\hat{\beta}^2c^2\,|x|^{-2\theta} - \hat{h}c\,|x|^{-\theta},
&\theta \in (0,1-\alpha), \\
\rule{0pt}{1.5em}\tfrac12\hat{\beta}^2\,c^*[\phi^2]\,\hat\delta_0(x) 
- \hat{h}c\,|x|^{-\theta},
&\theta \in (1-\alpha, 2(1-\alpha)) \\
\rule{0pt}{1.5em}\big\{\tfrac12\hat{\beta}^2\,c^*[\phi^2]-\hat{h}\,c^*[\phi]\big\}\,\hat\delta_0(x),  
&\theta \in (2(1-\alpha),\infty). 
\end{cases}
\end{equation}
We recall that the \emph{renormalized} Dirac-function $\hat\delta_0(\cdot)$ is the notation introduced 
in \eqref{eq:hatdelta} to make the time integral $\int_0^T \hat\psi_{\hat{\beta},\hat{h}}(X_t)\,\dd t$ in 
\eqref{eq:scalpart0} well-defined.

Once the convergence in \eqref{eq:scalpart0} is established, Theorems~\ref{thm:weak1}-\ref{thm:weak3} 
follow. Indeed, recalling the definitions of $\hat{F}^\ann$ in \eqref{eq:hatF1}, \eqref{eq:hatF2} and 
\eqref{eq:hatF3}, respectively, we can write
\begin{equation}
\label{eq:interchange}
\begin{aligned}
\hat{F}^\ann(\hat{\beta},\hat{h}) 
& = \lim_{T\to\infty} \frac{1}{T} \log \hat{Z}^\ann_{T,\hat{\beta},\hat{h}} 
= \lim_{T\to\infty} \frac{1}{T}
\log \left(\lim_{N\to\infty} 
Z^\ann_{TN, \beta_N, h_N} \right)\\
&= \lim_{N\to\infty} N \left( \lim_{T\to\infty} \frac{1}{TN} 
\log Z^\ann_{TN, \beta_N, h_N}  \right)\\
&= \lim_{N\to\infty} N F^\ann\big( \beta_N , h_N\big),
\end{aligned}
\end{equation}
where the interchange of the limits $T\to\infty$ and $N\to\infty$ is justified in
Appendix~\ref{app:interchange}.

To prove \eqref{eq:scalpart0}, we write $Z^\ann_{TN,\beta_N,h_N}$ as a series,
recall  \eqref{eq:Zann}. Since the  potential $\phi$ is assumed to be symmetric, 
i.e., $\phi(-x) = \phi(x)$ for all $x\in\Z$, we can write $\psi_{\beta,h}(S_n) = \psi_{\beta,h}(|S_n|)$,
and hence
\begin{equation}
\label{eq:ZannWDE0}
Z^\ann_{TN,\beta_N,h_N} = \e\left[\,\prod_{n=1}^{TN} 
\left\{1+\big(\ee^{\psi_{\beta_N,h_N}(|S_n|)}-1\big)\right\}\right]
= 1 + \sum_{k \in \N} C_{TN,k}
\end{equation} 
with
\begin{equation}
\label{eq:CNkdef0}
C_{TN,k} := \sum_{1 \leq n_1 < \cdots < n_k \leq TN} 
\e\left[\,\prod_{\ell=1}^k \chi_{\beta_N,h_N}(|S_{n_\ell}|)\right] , \qquad
\chi_{\beta,h}(x) := e^{\psi_{\beta,h}(x)}-1.
\end{equation}
We can write a similar decomposition for $\hat{Z}^\ann_{T,\hat{\beta},\hat{h}}$, namely,
\begin{equation}
\label{eq:hatZ0}
\hat{Z}^\ann_{T,\hat{\beta},\hat{h}} = 1 + \sum_{k \in \N}\hat{C}_{T,k}
\end{equation}
with
\begin{equation}
\label{eq:hatCexpr0}
\hat{C}_{T,k} := \frac{1}{k!} \, \hat{\e}\bigg[ 
\left( \int_0^T \dd t\,\hat\psi_{\hat{\beta},\hat{h}}(X_t)\right)^k \bigg].
\end{equation}
Formally, if we rewrite the $k$-th power of the integral as a $k$-fold integral and afterwards switch the 
integral and the expectation, then we get the expression
\begin{equation}
\label{eq:hatCexpr}
\begin{split}
\hat{C}_{T,k}
= \int\limits_{\substack{0 < t_1<\cdots<t_k < T \\
x_1, \ldots, x_k \in [0,\infty)}} 
\prod_{\ell=1}^k  \Big\{ g_{t_\ell - t_{\ell-1}}(x_{\ell-1}, x_\ell)
\, \hat\psi_{\hat{\beta},\hat{h}}\big(x_\ell\big) \Big\} \,
\dd t_1 \cdots \dd t_k\,
\dd x_1 \cdots \dd x_k
\end{split}
\end{equation}
with $t_0 := x_0 := 0$. This is justified by Fubini's theorem in the first regime in \eqref{eq:psi0def}, because 
$\hat\psi_{\hat{\beta},\hat{h}}(x)$ is a genuine function (see \eqref{eq:psi0def}). However, the second and 
third regime in \eqref{eq:psi0def} are more delicate because $\hat\psi_{\hat{\beta},\hat{h}}(x)$ contains the 
formal term $\hat\delta_0(x)$. We claim that \eqref{eq:hatCexpr} holds in these regimes as well, provided 
we use the interpretation
\begin{equation} 
\label{eq:recipe}
\text{``} \ g_{t_\ell - t_{\ell-1}}(x_{\ell-1}, x_\ell) \, \hat\delta_0(x_\ell) \, \dd x_\ell  \ \text{''} 
:= \hat g_{t_\ell - t_{\ell-1}}(x_{\ell-1}, 0) \, \delta_0(\dd x_\ell),
\end{equation}
where $\hat g$ is the function defined in \eqref{eq:hatgt} and $\delta_0$ is the usual Dirac measure at zero,
both of which are proper. We prove \eqref{eq:hatCexpr} in Section~\ref{sec:hatCexpr} below.

\begin{remark}
{\rm In the third regime, both terms in $\hat\psi_{\hat{\beta},\hat{h}}(x)$ contain $\hat\delta_0(x)$ (see 
\eqref{eq:psi0def}). Therefore \eqref{eq:hatCexpr} can be simplified, by the recipe in \eqref{eq:recipe}, 
to give
\begin{equation} 
\label{eq:hatCalte}
\begin{split}
\hat{C}_{T,k} 
&= \big\{\tfrac12\hat{\beta}^2\,c^*[\phi^2]-\hat{h}\,c^*[\phi]\big\}^k
\int\limits_{0 < t_1<\cdots<t_k < T} 
\prod_{\ell=1}^k \Big\{ \hat g_{t_\ell - t_{\ell-1}}(0,0) \Big\} \,
\dd t_1 \cdots \dd t_k \\
& = \big\{\tfrac12\hat{\beta}^2\,c^*[\phi^2]-\hat{h}\,c^*[\phi]\big\}^k
\int\limits_{0 < t_1<\cdots<t_k < T} 
\prod_{\ell=1}^k \frac{1}{(t_\ell - t_{\ell-1})^{1-\alpha}} \,
\dd t_1 \cdots \dd t_k  \\
& = \frac{\big[ \big\{\tfrac12\hat{\beta}^2\,c^*[\phi^2]-\hat{h}\,c^*[\phi]\big\}
T^\alpha \, \Gamma(\alpha) \big]^k}{\Gamma(\alpha k)},
\end{split}
\end{equation}
where the last equality can be seen to hold by the normalisation constant
of the Dirichlet distribution (see also \eqref{eq:Dirichlet}).
Therefore \eqref{eq:hatCexpr} is delicate only in the second regime.}
\end{remark}

We will prove in Section~\ref{sec:negligible-tails} that the series in \eqref{eq:ZannWDE0} and 
\eqref{eq:hatZ0} have
\emph{negligible tails}:
\begin{equation}
\label{eq:negligible-tails}
\begin{gathered}
\forall\, T \in (0,\infty) \ \ \forall\, \epsilon \in (0,1) \ \ \exists\, \bar{K} = \bar{K}(T,\epsilon) < \infty\colon \\
\limsup_{N\to\infty} \, \sum_{k > \bar{K}} C_{TN,k} < \epsilon \,, \qquad
\sum_{k > \bar{K}} \hat{C}_{T,k} < \epsilon.
\end{gathered}
\end{equation}
We will show in Section~\ref{sec:convCk} that 
\begin{equation}
\label{eq:convCk}
\forall\, k \in \N\colon \qquad
\lim_{N\to\infty} C_{TN,k} = \hat{C}_{T,k}.
\end{equation}
The last two equations combine to yield \eqref{eq:scalpart0} and complete the proof. \qed

%%%
\subsection{Proof of \eqref{eq:hatCexpr}}
\label{sec:hatCexpr}

For $\epsilon > 0$ we define a genuine function $\hat\delta_0^\epsilon(x)$ that is meant to 
approximate $\hat\delta_0(x)$ as $\epsilon \downarrow 0$ (recall \eqref{eq:hatLdef}--\eqref{eq:hatdelta}):
\begin{equation}
\label{eq:hatdeltaepsilon}
\hat\delta_0^\epsilon(x) := \frac{c_\alpha}{\epsilon^{2(1-\alpha)}} \, \ind_{(0,\epsilon)}(x).
\end{equation}
We also define an approximate version $\hat\psi^\epsilon(x)$ of $\hat\psi(x)$ in \eqref{eq:psi0def} by
setting (suppressing the dependence on $\hat{\beta},\hat{h}$)
\begin{equation}
\label{eq:psi0epsilondef}
\begin{split}
\hat\psi^\epsilon(x) := 
&\text{ the expression obtained from \eqref{eq:psi0def} after replacing $\hat\delta_0(x)$ by $\hat\delta_0^\epsilon(x)$},
\end{split}
\end{equation}
so that
\begin{equation}
\label{eq:psiconv}
\lim_{\epsilon \downarrow 0} \int_0^T \hat\psi^\epsilon(X_t) \, \dd t
= \int_0^T \hat\psi(X_t) \, \dd t \quad \text{in probability.} 
\end{equation}
Below we prove that this convergence \emph{also} holds in $L^k$, $k\in\N$. Then, via \eqref{eq:hatCexpr0}, 
we can write
\begin{equation}
\label{eq:hatCuno}
\begin{split}
\hat{C}_{T,k} &= \frac{1}{k!} \, \lim_{\epsilon\downarrow 0} \, \hat{\e}\bigg[ 
\left( \int_0^T \dd t\,\hat\psi^\epsilon(X_t)\right)^k \bigg] \\
&= \lim_{\epsilon\downarrow 0} \, 
\int\limits_{\substack{0 < t_1<\cdots<t_k < T \\
x_1, \ldots, x_k \in [0,\infty)}} 
\prod_{\ell=1}^k  \Big\{ g_{t_\ell - t_{\ell-1}}(x_{\ell-1}, x_\ell)
\, \hat\psi^\epsilon\big(x_\ell\big) \Big\} \,
\dd t_1 \cdots \dd t_k\,
\dd x_1 \cdots \dd x_k,
\end{split}
\end{equation}
which coincides precisely with our target equation \eqref{eq:hatCexpr} with the recipe in \eqref{eq:recipe} via
the characterization of $\hat g$ in \eqref{eq:hatgt0}--\eqref{eq:hatgt} (note that $x \mapsto g_t(x,y)$ is continuous 
on $[0,\infty)$).

To prove that \eqref{eq:psiconv} holds in $L^k$, $k\in\N$, it is enough to show that \emph{all 
moments of $\int_0^T \hat\psi^\epsilon(X_t) \, \dd t$ are uniformly bounded}:
\begin{equation}
\label{eq:unibound}
\forall\, n \in \N\colon \qquad
\sup_{\epsilon \in (0,1)}  \hat{\e} \bigg[ \bigg(\int_0^T \hat\psi^\epsilon(X_t) \, \dd t \bigg)^n \bigg] < \infty.
\end{equation}
Since $\hat\psi^\epsilon(x)$ is a genuine function, Fubini's theorem tells us that, for every $\epsilon > 0$,
\begin{equation}
\begin{split}
\frac{1}{n!} \, \hat{\e}_0 \bigg[ \bigg( \int_0^T \hat\psi^\epsilon(X_t) \, \dd t \bigg)^n \bigg]
& =  \int\limits_{0 < t_1<\cdots<t_n < T} 
\hat{\e}_0 \bigg[ \prod_{\ell=1}^n  \hat\psi^\epsilon\big(X_{t_\ell}\big) \bigg] \,
\dd t_1 \cdots \dd t_n.
\end{split}
\end{equation}
Combining \eqref{eq:psi0def} abd \eqref{eq:hatdeltaepsilon}--\eqref{eq:psi0epsilondef}, we can bound $|\hat\psi^\epsilon(x)| 
\le \bar{\psi}^\epsilon(x) := C_1 \, \hat\delta_0^\epsilon(x) + C_2 \, |x|^{-\gamma}$ for suitable constants $C_1, C_2$ 
and $\gamma \in \{\theta,2\theta\} < 2(1-\alpha)$. Since $\bar{\psi}^\epsilon$ is decreasing, and since
$\hat{\p}_0(X_t \in \cdot)$ is stochastically dominated by $\hat{\p}_x(X_t \in \cdot)$ for any $x \ge 0$,
we can bound
\begin{equation}
\label{eq:hatCdue}
\begin{split}
& \frac{1}{n!} \, \hat{\e}_0 \bigg[ \bigg( \int_0^T \hat\psi^\epsilon(X_t) \, \dd t \bigg)^n \bigg]
\le \int\limits_{0 < t_1<\cdots<t_n < T} 
\prod_{\ell=1}^n  \hat{\e}_0 \big[ \bar\psi^\epsilon\big(X_{t_\ell - t_{\ell-1}}\big) \big] \,
\dd t_1 \cdots \dd t_n \\
& \quad = \int\limits_{0 < t_1<\cdots<t_n < T} 
\prod_{\ell=1}^n 
\bigg\{ C_1 \, \frac{c_\alpha}{\epsilon^{2(1-\alpha)}}
\, \hat{\p}_0 (X_{t_\ell - t_{\ell-1}} < \epsilon )
+ C_2 \, \hat{\e}_0 \big[ |X_{t_\ell - t_{\ell-1}}|^{-\gamma} \big] \bigg\} \,
\dd t_1 \cdots \dd t_n \\
& \quad \le \int\limits_{0 < t_1<\cdots<t_n < T} 
\prod_{\ell=1}^n 
\bigg\{ \frac{C_1}{(t_\ell - t_{\ell-1})^{1-\alpha}}
+ \frac{C'_2}{(t_\ell - t_{\ell-1})^{\gamma/2}} \bigg\} \,
\dd t_1 \cdots \dd t_n,
\end{split}
\end{equation}
where the last inequality holds by \eqref{eq:gt}--\eqref{eq:PXepsilon}, for a suitable (and explicit) $C'_2 < \infty$
(note that $\hat{\e}_0[|X_{t_\ell - t_{\ell-1}}|^{-\gamma}] < \infty$ because $\gamma < 2(1-\alpha)$). The last 
expression is a finite constant, and so we have completed the proof of \eqref{eq:hatCexpr}. \qed

%%%

\subsection{Proof of \eqref{eq:negligible-tails}}
\label{sec:negligible-tails}

The second inequality in \eqref{eq:negligible-tails} follows from the bounds in \eqref{eq:hatCuno}--\eqref{eq:hatCdue}, namely, 
for some constant $C = C(T) < \infty$ (recall that $\gamma < 2(1-\alpha)$)
\begin{equation}
\label{eq:boChat}
\begin{split}
&\forall\, T \in (0,\infty) \quad \forall k \in \N\colon\\ 
&\hat C_{T,k} \le \int\limits_{0 < t_1<\cdots<t_k < T} 
\Bigg\{ \prod_{\ell=1}^k \frac{C}{(t_\ell - t_{\ell-1})^{1-\alpha}}
\Bigg\} \, \dd t_1 \cdots \dd t_k
= \frac{\big( C \, T^\alpha \, \Gamma(\alpha) \big)^k}{\Gamma(\alpha k)}
\end{split}
\end{equation}
(see also Appendix~\ref{app:finpartcont}). 

To prove the first inequality in \eqref{eq:negligible-tails}, we recall that, by \eqref{eq:CNkdef0},
\begin{equation} 
\label{eq:CNkdef0rep}
C_{TN,k} = \sum_{1 \leq n_1 < \cdots < n_k \leq TN} 
\e\left[\,\prod_{\ell=1}^k \chi_{\beta_N,h_N}(|S_{n_\ell}|)\right] \,, \qquad
\chi_{\beta_N,h_N}(x) = \ee^{\psi_{\beta_N,h_N}(x)}-1.
\end{equation}
In the sequel, $C, C' < \infty$ denote absolute constants, possibly depending on $T, \hat\beta,\hat h$,
that may change from line to line. Since $\beta_N, h_N \downarrow 0$, it follows from \eqref{eq:psi} that for
$N\to\infty$ (recall that $\phi$ is bounded):
\begin{equation} 
\label{eq:chibound}
\text{uniformly in } x\in\Z\colon \qquad
\chi_{\beta_N,h_N}(x) \sim
\psi_{\beta_N,h_N}(x) \sim
\tfrac12\beta_N^2 \, \phi(x)^2 \, - \, h_N \, \phi(x).
\end{equation}
Since $\lim_{|x|\to\infty} |x|^\theta\phi(x) = c \in (0,\infty)$ by \eqref{eq:phiscal}, we can bound
\begin{equation} 
\label{eq:ifoby}
\forall x \in \Z: \qquad
\big| \chi_{\beta_N,h_N}(|x|) \big| \le
\bar{\chi}_{\beta_N,h_N}(|x|) :=
C \, \big( \tfrac12\beta_N^2 \, (1+ |x|)^{-2\theta} \, + \, h_N \, (1+|x|)^{-\theta} \big).
\end{equation}

Next, recall the uniform upper bound in \eqref{eq:globalUB}. For any $\gamma \ne 2(1-\alpha)$ we can 
bound
\begin{equation}
\label{eq:teaches}
\begin{split}
\forall n \in \N: \qquad
\e\big[ (1+ |S_n|)^{-\gamma} \big] \le 
\frac{C}{n^{1-\alpha}}
\sum_{k\in\Z} (1+|k|)^{1-2\alpha-\gamma} \, \ee^{-\frac{|k|^2}{8n}}
\le \frac{C'}{n^{(1-\alpha) \wedge (\gamma/2)}}
\end{split}
\end{equation}
(where $a \wedge b := \min\{a,b\}$). Indeed,
\begin{itemize}
\item 
for $\gamma > 2(1-\alpha)$ we drop the exponential and note that $\sum_{k\in\Z} (1+|k|)^{1-2\alpha-\gamma} 
< \infty$;
\item 
for $\gamma < 2(1-\alpha)$ a Riemann sum approxmation shows that the sum can be bounded by a multiple 
of $n^{1-\alpha-\gamma/2}$ (note that $\int_0^\infty x^{1-2\alpha-\gamma} \, \ee^{-x^2/8} \, \dd x < \infty$).
\end{itemize}
It follows from \eqref{eq:ifoby}--\eqref{eq:teaches} that
\begin{equation}
\label{eq:ifoby2}
\forall\, n,N\in\N\colon \qquad
\e\big[\bar{\chi}_{\beta_N,h_N}(|S_n|) \big] 
\le C \, \bigg( \frac{\beta_N^2}{n^{(1-\alpha)\wedge \theta}}
\, + \, \frac{h_N}{n^{(1-\alpha)\wedge (\theta/2)}} \bigg).
\end{equation}
Setting $n = Nt$ with $t \in (0,T]$, and recalling that $\beta_N$, $h_N$ are given in \eqref{eq:betaN}--\eqref{eq:hN}, 
we obtain
\begin{equation} 
\label{eq:ebarchi}
\begin{split}
&\forall\, N\in\N, \ \forall t\, \in (0,T] \cap N^{-1}\N\colon \\
&\e\big[ \bar{\chi}_{\beta_N,h_N}(|S_{Nt}|)\big] 
\le \frac{C}{N} \, \bigg( \frac{1}{t^{(1-\alpha)\wedge \theta}}
+  \frac{1}{t^{(1-\alpha)\wedge (\theta/2)}} \bigg)
\le \frac{C'}{N} \, \frac{1}{t^{1-\alpha}}.
\end{split}
\end{equation}

Next, by \eqref{eq:CNkdef0rep} we can bound
\begin{equation} 
\label{eq:Crep}
\begin{split}
C_{TN,k} \le \sum_{1 \leq n_1 < \cdots < n_k \leq TN} 
\prod_{\ell=1}^k \e\left[\bar{\chi}_{\beta_N,h_N}(|S_{n_\ell-n_{\ell-1}}|)\right], 
\end{split}
\end{equation}
because the function $\bar{\chi}_{\beta_N,h_N}(|x|)$ is decreasing in $|x|$ and the law $\p(|S_n| \in 
\cdot \,|\, |S_{n'}| = m)$ stochastically dominates $\p(|S_n| \in \cdot \,|\, S_{n'} = 0)
= \p(|S_{n-n'}| \in \cdot)$ for all $m \geq 0$ (as can be seen via a coupling argument).
Then, setting $n_\ell = N t_\ell$, via \eqref{eq:ebarchi} we get
\begin{equation} 
\label{eq:Crep2}
\begin{split}
C_{TN,k} & \le \frac{1}{N^k} \,
\sum_{\substack{0 < t_1 < \cdots < t_k \leq T \\
\text{such that } N t_1, \ldots, N t_k \in \N}} 
\prod_{\ell=1}^k \frac{C}{(t_\ell-t_{\ell-1})^{1-\alpha}} \\
& \le \int\limits_{0 < t_1<\cdots<t_k < T} 
\Bigg\{ \prod_{\ell=1}^k \frac{C}{(t_\ell - t_{\ell-1})^{1-\alpha}}
\Bigg\} \, \dd t_1 \cdots \dd t_k.
\end{split}
\end{equation}
This is the same bound as in \eqref{eq:boChat}, and hence the first inequaltiy in \eqref{eq:negligible-tails} is 
proved.\qed

%%%

\subsection{Proof of \eqref{eq:convCk}}
\label{sec:convCk}

Intuitively, we can get \eqref{eq:convCk} by substituting the asymptotic relation \eqref{eq:chibound} 
into the definition \eqref{eq:CNkdef0rep} of $C_{TN,k}$ and performing a Riemann sum approximation,
recalling \eqref{eq:phiscal}. However, some care is required to properly implement this strategy.

\medskip\noindent
{\bf 1.} Recall from \eqref{eq:chibound} that $\chi_{\beta_N,h_N}(x) \sim \frac{\beta_N^2}{2} \, \phi(x)^2  
- h_N \, \phi(x)$ as $N\to\infty$, and by \eqref{eq:phiscal} the terms $\phi(x)^2$ and $\phi(x)$ decay 
polynomially as $|x|^{-\gamma}$ as $|x|\to\infty$, with $\gamma = 2\theta$ and $\gamma = \theta$, 
respectively. Fix $\epsilon > 0$ small and introduce an approximation $\chi^\epsilon_{\beta_N,h_N}(x)$ 
of $\chi_{\beta_N,h_N}(x)$ in which each term is restricted to a relevant range: either $|x| \approx 1$ 
(more precisely, $|x| \le \frac{1}{\epsilon}$) or $|x| \approx \sqrt{N}$ (more precisely, $\epsilon \sqrt{N} 
\le |x| \le \frac{1}{\epsilon} \sqrt{N}$), depending on the decay exponent $\gamma$. Indeed, the bound 
in \eqref{eq:teaches} on $\e\big[ (1+ |S_N|)^{-\gamma} \big] \le$ tells us that
\begin{itemize}
\item 
for $\gamma > 2(1-\alpha)$ the relevant contribution comes from $|x| = |S_N| \approx 1$;
\item 
for $\gamma < 2(1-\alpha)$ the relevant contribution comes from $|x| = |S_N| \approx \sqrt{N}$.
\end{itemize}
This motivates the following definition:
\begin{equation}
\label{eq:hatpsi}
\chi^\epsilon_{\beta_N,h_N}(x) :=
\begin{cases}
\Big\{\tfrac12\beta_N^2 \, \phi(x)^2 \, - \, h_N \, \phi(x) \Big\}
\ind_{\{\epsilon \sqrt{N} \le |x| \le \frac{1}{\epsilon} \sqrt{N}\}},
&\theta \in (0,1-\alpha), \\
\rule{0pt}{1.7em}\tfrac12\beta_N^2 \, \phi(x)^2 
\ind_{\{|x| \le \frac{1}{\epsilon}\}} - h_N \, \phi(x) \, \ind_{\{\epsilon \sqrt{N} \le |x| \le \frac{1}{\epsilon} \sqrt{N}\}},
&\theta \in (1-\alpha, 2(1-\alpha)), \\
\rule{0pt}{1.7em}\Big\{\tfrac12\beta_N^2 \, \phi(x)^2 
- h_N \, \phi(x) \Big\} \ind_{\{|x| \le \frac{1}{\epsilon}\}},
&\theta \in (2(1-\alpha),\infty). 
\end{cases}
\end{equation}
We correspondingly define an approximation $C^\epsilon_{TN,k}$ of $C_{TN,k}$ (see \eqref{eq:CNkdef0rep}):
\begin{equation} 
\label{eq:CNkappr}
C^\epsilon_{TN,k} = \sum_{1 \leq n_1 < \cdots < n_k \leq TN} 
\e\left[\,\prod_{\ell=1}^k \chi^\epsilon_{\beta_N,h_N}(|S_{n_\ell}|)\right].
\end{equation}
This allows us to split the proof of \eqref{eq:convCk} in two parts: For \emph{fixed} $k\in\N$,
\begin{equation}
\label{eq:splittwo}
\lim_{\epsilon \downarrow 0} \ \limsup_{N\to\infty} \
\big| C_{TN,k} - C^\epsilon_{TN,k} \big| = 0,
\qquad
\lim_{\epsilon \downarrow 0} \ \lim_{N\to\infty} \ C^\epsilon_{TN,k} = \hat{C}_{T,k}.
\end{equation}

\medskip\noindent
{\bf 2.} The first claim in \eqref{eq:splittwo} is a consequence of the bounds that we have derived in 
Section~\ref{sec:negligible-tails} for the proof of \eqref{eq:negligible-tails}. Indeed, we can use a 
telescopic argument based on the bound
\begin{equation}
\bigg|\prod_{\ell=1}^k a_\ell - \prod_{\ell=1}^k b_\ell\bigg| \,\le\, \sum_{i=1}^k \ 
\bigg( \prod_{\ell=1}^{i-1} |a_\ell| \bigg)  |a_i - b_i|  \bigg( \prod_{\ell=i+1}^k |b_\ell| \bigg) 
\end{equation}
to replace $a_\ell := \chi_{\beta_N,h_N}(S_{n_\ell})$ by $b_\ell := \chi^\epsilon_{\beta_N,h_N}(S_{n_\ell})$ 
in the definition \eqref{eq:CNkdef0rep} of $C_{TN,k}$. Note that both $|a_\ell|$ and $|b_\ell|$ are bounded by 
$\bar{\chi}_{\beta_N,h_N}(|S_{n_\ell}|)$ defined in \eqref{eq:ifoby}. Similarly, we can bound $|a_i - b_i|
= |(\chi_{\beta_N,h_N} - \chi^\epsilon_{\beta_N,h_N})(S_{n_i})| \le \Delta\bar{\chi}^\epsilon_{\beta_N,h_N}
(|S_{n_i}|)$, where we define $\Delta\bar{\chi}^\epsilon_{\beta_N,h_N}(|x|)$ to be $\bar{\chi}_{\beta_N,h_N}(|x|)$ 
in \eqref{eq:ifoby} with the terms $(1+ |x|)^{-2\theta}$ and $(1+ |x|)^{-\theta}$ replaced as follows:
\begin{equation}
\label{eq:replacement}
(1+|x|)^{-\gamma} \ \rightsquigarrow \
\begin{cases}
(1+|x|)^{-\gamma} \, \ind_{\{|x| > \frac{1}{\epsilon}\}}, 
&\gamma > 2(1-\alpha) \\
(1+|x|)^{-\gamma} \, \ind_{\{|x| < \epsilon \sqrt{N}\}
\cup \{|x| > \frac{1}{\epsilon} \sqrt{N}\}}, 
&\gamma < 2(1-\alpha). 
\end{cases}
\end{equation}
This leads to
\begin{equation} 
\label{eq:thisleadsto}
\big| C_{TN,k} - C^\epsilon_{TN,k} \big| \le\sum_{i=1}^k \, 
\sum_{1 \leq n_1 < \cdots < n_k \leq TN} 
\e\left[\, \Delta\bar\chi^\epsilon_{\beta_N,h_N}(|S_{n_i}|)
\times \prod_{\ell \in \{1,\ldots,k\} \setminus \{i\}} 
\bar\chi_{\beta_N,h_N}(|S_{n_\ell}|)\right].
\end{equation}

\medskip\noindent
{\bf 3.} Recall the bound \eqref{eq:ifoby2} for $\e[\bar\chi_{\beta_N,h_N}(|S_{n}|)]$. A similar but improved 
bound holds for $\e\big[\Delta\bar{\chi}^\epsilon_{\beta_N,h_N}(|x|)\big]$, namely,
\begin{equation}
\label{eq:ifoby3}
\begin{split}
\forall\, n,N\in\N\colon \qquad
& \e\big[\Delta\bar{\chi}^\epsilon_{\beta_N,h_N}(|S_n|) \big] \le
C \, \eta_\epsilon \, \bigg( \frac{\beta_N^2}{n^{(1-\alpha)\wedge \theta}}
\, + \, \frac{h_N}{n^{(1-\alpha)\wedge (\theta/2)}} \bigg)\\
& \text{with} \quad \lim_{\epsilon\downarrow 0} \eta_\epsilon = 0.
\end{split}
\end{equation}
Indeed, by \eqref{eq:teaches} and the lines following it, we can choose
\begin{equation}
\eta_\epsilon := \begin{cases}
\sum_{k > 1/\epsilon}  (1+|k|)^{1-2\alpha-\gamma}, 
&\gamma > 2(1-\alpha) \\
\rule{0pt}{1.5em}\int_{(0,\epsilon) \cup (1/\epsilon,\infty)} 
x^{1-2\alpha-\gamma} \, e^{-x^2/8} \, \dd x, 
&\gamma < 2(1-\alpha). 
\end{cases}
\end{equation}
Consequently, arguing as in \eqref{eq:ebarchi}--\eqref{eq:Crep}, we get the following improvement of 
\eqref{eq:Crep2}:
\begin{equation} 
\label{eq:Crep3}
\begin{split}
\big| C_{TN,k} - C^\epsilon_{TN,k} \big| & \le
k \, \eta_\epsilon \, \int\limits_{0 < t_1<\cdots<t_k < T} 
\Bigg\{ \prod_{\ell=1}^k \frac{C}{(t_\ell - t_{\ell-1})^{1-\alpha}}
\Bigg\} \, \dd t_1 \cdots \dd t_k 
\end{split}
\end{equation}
(set $n_0 = 0$, $x_0 = 0$), from which the first equation in \eqref{eq:splittwo} readily follows.
With similar arguments we can show that in \eqref{eq:CNkappr} we can further restrict the sum to
$n_\ell - n_{\ell - 1} \ge \epsilon N$ for all $\ell=1,\ldots, k$, with a negligible error as $N\to\infty$ 
followed by $\epsilon \downarrow 0$ (we omit the details). We can thus update the definition 
\eqref{eq:CNkappr} of $C^\epsilon_{TN,k}$, where we also sum over space variables:
\begin{equation} 
\label{eq:CNkappr2}
C^\epsilon_{TN,k} = \sum_{\substack{1 \leq n_1 < \cdots < n_k \leq TN: \\[0.1cm]
n_\ell - n_{\ell - 1} \ge \epsilon N \ \forall\, \ell=1,\ldots, k\\[0.1cm]
x_1, \ldots, x_k \in \N_0}} \
\prod_{\ell=1}^k 
\p_{x_{\ell-1}}(|S_{n_\ell-n_{\ell-1}}| = x_\ell)
\, \chi^\epsilon_{\beta_N,h_N}(x_\ell).
\end{equation}

\medskip\noindent
{\bf 4.} We are left with proving the second claim in \eqref{eq:splittwo}. To do so, we distinguish between the three
regimes in \eqref{eq:hatpsi}.

\medskip\noindent
\paragraph{$\bullet$ First regime.}
In the regime $\theta \in (0,1-\alpha)$, the function $\chi^\epsilon_{\beta_N,h_N}(x_\ell)$ restricts all 
variables $x_\ell$ in the diffusive range $\{\epsilon \sqrt{N} \le x \le \frac{1}{\epsilon} \sqrt{N}\}$ 
(``intermediate heights''). This allows us to apply the local limit theorem in \eqref{eq:llt}, which we may rewrite 
as follows: Uniformly in $t \in [\epsilon, T]$ and $\epsilon \le z', z \le 1/\epsilon$,
\begin{equation} 
\label{eq:lltbis}
\p_{\sqrt{N} z'}(|S_{Nt}| = \sqrt{N}z) 
\sim \frac{2}{\sqrt{N}} \, g_t(z', z)\,
\ind_{\{\text{$\sqrt{N}z-\sqrt{N}z'-Nt$ is even}\}}, 
\qquad N \to \infty,
\end{equation}
where we recall that $g_t(z',z) = g_1(z'/\sqrt{t}, z/\sqrt{t})/\sqrt{t}$. (For notational simplicity, we will ignore the 
parity constraint and remove the compensating factor $2$.) Note that $\phi(x) \sim c /  x^\theta$ as $x \to \infty$ 
(see \eqref{eq:phiscal}). Therefore, by \eqref{eq:betaN}--\eqref{eq:psi0def} and \eqref{eq:hatpsi}, uniformly in 
$\epsilon \le z \le 1/\epsilon$,
\begin{equation}
\begin{split}
\chi^\epsilon_{\beta_N,h_N}(\sqrt{N} z) 
&\sim \big\{\tfrac12 \beta_N^2 \, c^2 \, (\sqrt{N} z)^{-2\theta} 
- h_N \, c \, (\sqrt{N} z)^{-\theta} \big\} 
\sim \frac{\hat\psi_{\hat\beta,\hat h}(z)}{N}.
\end{split}
\end{equation}
Introducing macroscopic space-time variables $n_\ell = N t_\ell$, $x_\ell = \sqrt{N} z_\ell$, we get from 
\eqref{eq:CNkappr2} that
\begin{equation}
\begin{split}
C^\epsilon_{TN,k} 
&\sim \sum_{\substack{0 < t_1 < \cdots < t_k \leq T:\\[0.1cm]
N t_\ell \in \N \text{ and } t_\ell-t_{\ell-1} \ge \epsilon \ \forall\, \ell \\[0.1cm]
z_1, \ldots, z_k \in [0,\infty): \\[0.1cm]
\sqrt{N}z_\ell \in \Z
\text{ and } \epsilon \le z_\ell \le 1/\epsilon \ \forall\, \ell}} \
\prod_{\ell=1}^k \frac{g_{t_\ell-t_{\ell-1}}(z_{\ell-1}, z_\ell)}{\sqrt{N}}
\, \frac{\hat\psi_{\hat\beta,\hat h}(z_\ell)}{N} \\
& \xrightarrow[\, N\to\infty \, ]{} \int\limits_{\substack{0 < t_1 < \cdots < t_k \leq T:\\[0.1cm]
t_\ell-t_{\ell-1} \ge \epsilon \ \forall\, \ell \\[0.1cm]
z_1, \ldots, z_k \in [0,\infty)\colon \\[0.1cm]
\epsilon \le z_\ell \le 1/\epsilon \ \forall\, \ell}} 
\prod_{\ell=1}^k 
\big\{ g_{t_\ell-t_{\ell-1}}(z_{\ell-1}, z_\ell) \, \hat\psi_{\hat\beta,\hat h}(z_\ell) \big\} 
\, \dd t_1 \, \ldots \, \dd t_k \, \dd z_1 \, \ldots
\, \dd z_k.
\end{split}
\end{equation}
When we let $\epsilon \downarrow 0$, the last integral converges to $\hat C_{TN,k}$ (see \eqref{eq:hatCexpr}). 
This proves the second claim in \eqref{eq:splittwo}.

\medskip\noindent
\paragraph{$\bullet$ Third regime.}
In the regime $\theta \in (2(1-\alpha),\infty)$, the function $\chi^\epsilon_{\beta_N,h_N}(x_\ell)$
restricts all variables $x_\ell$ to the $O(1)$ range $\{x \le \frac{1}{\epsilon} \}$ (``low heights'').
This allows us to apply the local limit theorem in \eqref{eq:Besseltail}, which we may rewrite as follows:
Uniformly in $t \in [\epsilon, T]$ and $0 \le x', x \le 1/\epsilon$,
\begin{equation}
\label{eq:Besseltailbis}
\p_{x'}(|S_{Nt}| = x) \sim  \frac{2\,c(x)}{N^{1-\alpha} \, t^{1-\alpha}}\,
\ind_{\{\text{$x-x'-Nt$ is even}\}}, \qquad N \to \infty,
\end{equation}
where we recall that $\hat g_1(x'/\sqrt{n},0) \sim \hat g_1(0,0) = 1$ (see \eqref{eq:hatgt}).
(We again ignore the parity constraint and remove the compensating factor $2$.) Moreover, it follows 
from \eqref{eq:betaN}--\eqref{eq:hN} and \eqref{eq:hatpsi} that
uniformly for $x \le 1/\epsilon$,
\begin{equation}
\begin{split}
\chi^\epsilon_{\beta_N,h_N}(x) 
&= \bigg\{\tfrac12 \beta_N^2 \, \phi(x)^2
- h_N \, \phi(x) \bigg\} 
\sim \frac{\tfrac12\hat\beta^2\, \phi(x)^2
- \hat h \, \phi(x)}{N^\alpha}.
\end{split}
\end{equation}
Introducing the macroscopic time variables $n_\ell = N t_\ell$, while keeping the microscopic space
variables $x_\ell$, we get by \eqref{eq:CNkappr2} that
\begin{equation}
\begin{aligned}
&C^\epsilon_{TN,k}\\ 
&\sim \sum_{\substack{0 < t_1 < \cdots < t_k \leq T\colon\\[0.1cm]
N t_\ell \in \N \text{ and } t_\ell-t_{\ell-1} \ge \epsilon \ \forall\, \ell \\[0.1cm]
x_1, \ldots, x_k \in \N_0\colon \\[0.1cm]
x_\ell \le 1/\epsilon \ \forall\, \ell}} \
\prod_{\ell=1}^k 
\frac{c(x_\ell)}{N^{1-\alpha} \, (t_\ell-t_{\ell-1})^{1-\alpha}}
\, \frac{\tfrac12\hat\beta^2 \, \phi(x_\ell)^2
- \hat h \, \phi(x_\ell)}{N^\alpha} \\
& \xrightarrow[\, N\to\infty \, ]{} \int\limits_{\substack{0 < t_1 < \cdots < t_k \leq T\colon\\[0.1cm]
t_\ell-t_{\ell-1} \ge \epsilon \ \forall\, \ell }} 
\prod_{\ell=1}^k 
\frac{\big\{\tfrac12\hat\beta^2 \, \sum_{x \le 1/\epsilon} c(x) \phi(x)^2
- \hat h \, \sum_{x \le 1/\epsilon} c(x) \phi(x)\big\}}{(t_\ell-t_{\ell-1})^{1-\alpha}}
\, \dd t_1 \, \ldots \, \dd t_k.
\end{aligned}
\end{equation}
When we let $\epsilon \downarrow 0$, the term in brackets converges to
\begin{equation}
\tfrac12\hat\beta^2 \, \sum_{x \in\N_0} c(x) \phi(x)^2
- \hat h \, \sum_{x \in\N_0} c(x) \phi(x) =: 
\tfrac12\hat\beta^2\, c^*[\phi] - \hat h \, c^*[\phi].
\end{equation}
Hence $C^\epsilon_{TN,k} $ converges to $\hat C_{TN,k}$ (see \eqref{eq:hatCalte}). This proves 
the second claim in \eqref{eq:splittwo}.

\medskip\noindent
\paragraph{$\bullet$ Second regime.}
In the regime $\theta \in (1-\alpha, 2(1-\alpha))$, the function $\chi^\epsilon_{\beta_N,h_N}(x_\ell)$ 
is the sum of two terms with different restrictions, namely, $\{x \le \frac{1}{\epsilon} \}$ and 
$\{\epsilon \sqrt{N} \le x \le \frac{1}{\epsilon} \sqrt{N}\}$. Expanding the product over $\chi^\epsilon_{\beta_N,h_N}(
x_\ell)$ in \eqref{eq:CNkappr2}, we obtain $2^k$ different terms, where every variable $x_\ell$ is subject to 
one of the two restrictions. Each of these terms can be analysed by arguing precisely as in the previous
regimes: 
\begin{itemize}
\item[-] 
for $\{\epsilon \sqrt{N} \le x_\ell \le \frac{1}{\epsilon} \sqrt{N}\}$ we apply the local limit theorem 
in \eqref{eq:lltbis}, introducing the macroscopic space variable $z_\ell \in [0,\infty)$ defined by $x_\ell 
= \sqrt{N} z_\ell$;
\item[-] 
for $\{x_\ell \le \frac{1}{\epsilon} \}$ we apply the local limit theorem in \eqref{eq:Besseltailbis}, keeping the 
microscopic variable $x_\ell \in \N_0$.
\end{itemize}
The choice of $\beta_N$, $h_N$ in \eqref{eq:betaN}--\eqref{eq:hN} ensure that the appropriate Riemann factor 
appears, so that the sum defining $C^\epsilon_{NT,k}$ converges to the integral in \eqref{eq:hatCexpr} as 
$N\to\infty$ followed by $\epsilon \downarrow 0$. This integral defines $\hat C_{NT,k}$, with the recipe 
in \eqref{eq:recipe}. We omit the details, which are notationally tedious because of the need to specify 
the restriction for each variable $x_\ell$, but are ultimately straightforward.

%%%%%%%%%%% Appendix %%%%%%%%%%%%%%%%

\appendix

%%%%%%%%%%%% Appendix A %%%%%%%%%%%%%%

\section{Equivalence}
\label{appA}

In Appendices~\ref{app:exceed}--\ref{app:compareann} we prove \eqref{eq:exceed}, 
\eqref{eq:compare} and \eqref{eq:compareann}, respectively. In Appendix~\ref{app:interchange} 
we show that the limits $N\to\infty$ and $T\to\infty$ in \eqref{eq:interchange} may be interchanged. 

%%%

\subsection{Proof of \eqref{eq:exceed}} 
\label{app:exceed}

For $N\in\N_0$ and $x \in \N_0$, define
\begin{equation}
f_x(N) := \p_x\big(S_n \geq x\,\,\forall\,0 < n < N, S_N=x\big).
\end{equation}
(If $S$ has period 2, then replace $N$ by $2N$.) By superadditivity,
\begin{equation}
\lim_{N\to\infty} \frac{1}{N} \log f_x(N) =: C_x \in (-\infty,0] \quad \text{exists}.
\end{equation}
For $z \in [0,\infty)$ and $x \in \N_0$, define
\begin{equation}
F_x(z) := \sum_{N \in \N_0} z^N f_x(N). 
\end{equation}
Then, clearly,
\begin{equation}
C_x = 0 \quad \Longleftrightarrow \quad F_x(1+\epsilon) = \infty\,\,\forall\,\epsilon>0. 
\end{equation}
By \eqref{eq:astau}, we have $C_0=0$. Below we show that this implies $C_x=0\,\,\forall\, 
x\in\N_0$. The proof is by induction on $x$.

Any path that starts at $x$, ends at $x$ and does not go below $x$ can be cut into pieces 
that zig-zag between $x$ and $x+1$ and pieces that start at $x+1$, end at $x+1$ and do 
not go below $x+1$. Hence we have
\begin{equation}
F_x(z) = \frac{r_xzQ_x(z)}{1-F_{x+1}(z)Q_x(z)}, \qquad Q_x(z) := \frac{1}{1-zr_x},
\quad  r_x := p_{x,x+1}p_{x+1,x}>0,
\end{equation}
where $p_{x,x+1}:=\p(S_1=x+1 \mid S_0=x)$. We know that $F_0(1+\epsilon) = \infty\,\,\forall\,
\epsilon>0$. We show that this implies $F_x(1+\epsilon) = \infty\,\,\forall\,\epsilon>0$ for every 
$x\in\N$. The proof is by induction on $x$.

Fix $x\in\N_0$ and suppose that $F_x(1+\epsilon) = \infty\,\,\forall\,\epsilon>0$. We argue by 
contradiction. Suppose that $F_{x+1}(1+\epsilon)<\infty$ for $\epsilon$ small enough. Because 
$r_x>0$ and $Q_x(1+\epsilon)<\infty$ for $\epsilon$ small enough, it follows that $F_{x+1}(1+\epsilon)
Q_x(1+\epsilon) \geq 1\,\,\forall\,\epsilon>0$, and by continuity that $F_{x+1}(1)Q_x(1) \geq 1$. 
To get the contradiction it therefore suffices to show that $F_{x+1}(1)Q_x(1)<1$. Now, because 
$(S_n)_{n\in\N_0}$ is recurrent, we have $F_{x+1}(1) = p_{x+1,x+2}$. Because $Q_x(1) = 
1/(1-r_x)$, it follows that
\begin{equation}
F_{x+1}(1)Q_x(1) = \frac{p_{x+1,x+2}}{1-r_x} = \frac{p_{x+1,x+2}}{1-p_{x+1,x}p_{x,x+1}}
= \p_{x+1}(S_n \geq x\,\,\forall\,n \in \N_0) < 1,
\end{equation}
where for both the last equality and the inequality we again use recurrence.  

To prove \eqref{eq:exceed}, note that for any $0 \leq M \leq N$,
\begin{equation}
\begin{aligned}
\p_0(S_n \geq M\,\,\forall\,M < n \leq N) 
&= \p_0(S_n=n\,\,\forall\,0 < n \leq M)\,\p_M(S_n \geq M\,\,\forall\,M < n \leq N)\\ 
&\geq \left[\prod_{x=0}^{M-1} p_{x,x+1}\right] \,f_M(N-M).
\end{aligned}
\end{equation} 
Hence
\begin{equation}
\liminf_{N\to\infty} \frac{1}{N} \log \p_0(S_n \geq M\,\,\forall\,M < n \leq N) \geq C_M 
\qquad \forall\,M\in\N_0.
\end{equation}
Since $C_M = 0$ for all $M\in\N_0$, this implies \eqref{eq:exceed}. 

%%%

\subsection{Proof of \eqref{eq:compare}}
\label{app:compare}

Note that $Z_{N,\beta,h}^{\omega,\rc} \leq Z_{N,\beta,h}^\omega$ for all $N\in\N_0$. Hence we 
only need to show that $Z_{N,\beta,h}^{\omega} \leq \ee^{o(N)}\,Z_{N,\beta,h}^{\omega,\rc}$ 
$\bbP$-a.s. Define the constrained partition function
\begin{equation}
\label{eq:Znxy}
Z_{N,\beta,h}^{\omega,x,y}  
:= \e_x\left[\ee^{\sum_{n=1}^N (\beta\omega_n-h)\phi(S_n)}\,1_{\{S_n = y\}}\right],
\qquad N \in \N_0,\,x,y \in \Z. 
\end{equation}
Then
\begin{equation}
Z_{N,\beta,h}^\omega = \sum_{y \in \Z} Z_{N,\beta,h}^{\omega,0,y},
\qquad Z_{N,\beta,h}^{\omega,\rc} = Z_{N,\beta,h}^{\omega,0,0}.
\end{equation}
By \eqref{eq:phidecay}, for every $\epsilon>0$ there is an $M \in\N_0$ such that 
\begin{equation}
\label{eq:phiapprox}
|\phi(x)-\gamma^+| \leq \epsilon, \quad x \geq M, 
\qquad 
|\phi(x)-\gamma^-| \leq \epsilon, \quad x \leq -M.
\end{equation} 
where we abbreviate $\gamma^\pm = \lim_{x\to\pm\infty} \phi(x)$ and recall that $\gamma^+=0$,
$\gamma^- \in [0,\infty)$ according to \eqref{eq:phidecay}. For $N \geq M$, split
\begin{equation}
Z_{N,\beta,h}^\omega = \mathrm{I} + \mathrm{II} + \mathrm{III}
\end{equation}
with
\begin{equation}
\mathrm{I} = \sum_{y \geq 2M} Z_{N,\beta,h}^{\omega,0,y}, \qquad
\mathrm{II} = \sum_{y \leq -2M} Z_{N,\beta,h}^{\omega,0,y}, \qquad
\mathrm{III} = \sum_{-2M<y<2M} Z_{N,\beta,h}^{\omega,0,y}.
\end{equation}  
We will show that all three terms are bounded by $\ee^{o(N)}Z_{N,\beta,h}^{\omega,0,0}$.
For ease of notation we will pretend that the Markov chain has period $1$. This is easily
fixed when the period is $2$. 

Consider $\mathrm{I}$. Let $\sigma_N = \max\{0 \leq n \leq N\colon\,S_n = M\}$, and split
\begin{equation}
\mathrm{I} = \sum_{M \leq m \leq N-M} \e_0\left[\ee^{\sum_{n=1}^N (\beta\omega_n-h)
\phi(S_n)}\,1_{\{S_N \geq 2M\}}\,1_{\{\sigma_N=m\}}\right].
\end{equation}
Replace $(S_m,\ldots,S_N)$ by an $(N-M-m)$-step path from $M$ to $M$ that is 
everywhere $\geq M$, followed by an $M$-step downward path $(M,\ldots,0)$. 
The cost of this replacement is at most 
\begin{equation}
\exp\Big(\big[2\epsilon (N-M-m) + 2M\|\phi\|_\infty\big]\,(\beta\,\Omega_N(\omega)+|h|)\Big)
\end{equation} 
for the weight factor, with $\Omega_N(\omega) = \max_{1 \leq n \leq N} \omega_n$, 
and at most
\begin{equation}
\frac{\p_M(S_n \geq M\,\,\forall\, 0<n<N-M, S_{N} \geq 2M)}
{\p_M(S_n \geq M\,\,\forall\, 0<n<N-M-m, S_{N-M-m}=M)\,\Pi_M}
\end{equation}  
for the path probability, with $\Pi_M := \prod_{x=M}^1 p_{x,x-1}$. The probability in the denominator 
equals $f_M(N-M-m)$, and so we get
\begin{equation}
\mathrm{I} \leq \ee^{[2\epsilon N + 2M\|\phi\|_\infty]\,(\beta\,\Omega_N(\omega)+|h|)}
\frac{\Pi_M^{-1}}{f_M(N-M-m)}\,Z_{N,\beta,h}^{\omega,0,0}.
\end{equation}
Since $M$ is fixed, $\Pi_M>0$, $\|\phi\|_\infty<\infty$, $\Omega_N = O(\log N)$ $\bbP$-a.s.\ 
and $\max_{1 \leq n \leq N} 1/f_M(n) = \ee^{o(N)}$, we get $\mathrm{I} = \ee^{o(N)}
Z_{N,\beta,h}^{\omega,0,0}$ $\bbP$-a.s.

The argument for $\mathrm{II}$ is similar. For $\mathrm{III}$ we replace $(S_{N-M},\ldots,S_N)$ 
by an $M$-step path from $S_{N-M}$ to $0$ at a finite cost (depending on $M$ and $\Omega_N$).   

\subsection{Proof of \eqref{eq:compareann}}
\label{app:compareann}
 
Copy the argument in Appendix~\ref{app:compare} starting from the analogue of \eqref{eq:Znxy}:
\begin{equation}
Z_{N,\beta,h}^{\ann,x,y}  
:= \e_x\left[\ee^{\sum_{n=1}^N \psi_{\beta,h}(S_n)}\,1_{\{S_n = y\}}\right],
\qquad N \in \N_0,\,x,y \in \Z. 
\end{equation}
Use that $\|\psi_{\beta,h}\|_\infty < \infty$, and that for every $\epsilon>0$ there is an $M \in\N_0$ 
such that
\begin{equation}
\label{eq:psiapprox}
|\psi_{\beta,h}(x)-\eta^+| \leq \epsilon, \quad x \geq M,
\qquad 
|\psi_{\beta,h}(x)-\eta^-| \leq \epsilon, \quad x \leq -M,
\end{equation} 
with $\eta^\pm = \psi_{\beta,h}(\gamma^\pm) = \log M(\beta\gamma^\pm)-h\gamma^\pm$ (recall 
\eqref{eq:psi}).

%%%

\subsection{Interchange of limits}
\label{app:interchange}

The following two lemmas, which give us a sandwich for the annealed free energies in the discrete 
and in the continuous model, are the key to showing that the limits in \eqref{eq:interchange} may 
be interchanged.

\begin{lemma}
\label{lem:Fannalt}
For $M,N\in\N_0$, $\beta \in (0,\infty)$ and $h \in \R$, define
\begin{equation}
\label{eq:Zann+-def}
\begin{aligned}
Z^{\ann,-}_{M,N,\beta,h} &:= \min_{|x| \leq M} \sum_{|y| \leq M}
Z^{\ann,x,y}_{N,\beta,h},\\
Z^{\ann,+}_{N,\beta,h} &:= \sup_{x\in\Z} \sum_{y\in\Z} 
\max_{0 \leq N' \leq N} Z^{\ann,x,y}_{N',\beta,h}.
\end{aligned}
\end{equation}
Assume that $\lim_{|x| \to \infty} \phi(x) = 0$ (recall \eqref{eq:phiscal}). Then, for all $\beta \in 
(0,\infty)$ and $h \in \R$,
\begin{equation}
\label{eq:Fannsand}
\lim_{N\to\infty} \frac{1}{N} \log Z^{\ann,-}_{M,N,\beta,h}
= F^\ann(\beta,h) = \lim_{N\to\infty} \frac{1}{N} \log Z^{\ann,+}_{N,\beta,h}
\qquad \forall\,M \in \N_0.
\end{equation}
\end{lemma} 

\begin{proof}
Note that $Z^{\ann,-}_{M,N,\beta,h} \leq Z^{\ann,+}_{N,\beta,h}$ for all $M,N \in \N_0$. Hence 
we only need to show that $Z^{\ann,+}_{M,N,\beta,h} \leq \ee^{o(N)} Z^{\ann,-}_{N,\beta,h}$ 
as $N\to\infty$ for all $M\in\N_0$. Indeed, this will imply \eqref{eq:Fannsand} because $F^\ann
(\beta,h) = \lim_{N\to\infty} \frac{1}{N} \log Z^{\ann,0,0}_{N,\beta,h}$, as proved in 
Appendix~\ref{app:compareann}. In what follows we fix $\beta,h$ and abbreviate 
$\cH_N(S):=\sum_{n=1}^N \psi_{\beta,h}(S_n)$. 

Recall \eqref{eq:psiapprox}, where $\eta^\pm = 0$ because $\gamma^\pm = \lim_{x \to 
\pm\infty} \phi(x) = 0$ by assumption. Define $\tau_M:= \min\{n\in\N_0\colon\,|S_n| \leq M\}$ 
and split
\begin{equation}
\label{eq:xsplit}
\e_x\left[\ee^{\cH_N(S)}\right] = \mathrm{I}_x + \mathrm{II}_x
\end{equation}
with
\begin{equation}
\mathrm{I}_x = \e_x\left[\ee^{\cH_N(S)} 1_{\{\tau_M > N\}}\right],
\qquad 
\mathrm{II}_x = \e_x\left[\ee^{\cH_N(S)} 1_{\{\tau_M \leq N\}}\right].
\end{equation}
For $x \geq M$ and $x \leq -M$,
\begin{equation}
\label{eq:Ixbds}
\begin{aligned}
\mathrm{I}_x &\leq \ee^{\epsilon N} \p_x(\tau_M>N) \leq \ee^{\epsilon N},\\
\mathrm{II}_x &\leq \e_x\left[\ee^{\epsilon\tau_M} 1_{\{\tau_M \leq N\}} \max_{z = \pm M} 
\e_z\left[\ee^{\cH_{N-\tau_M}(S)}\right] \right]\\
&\leq \ee^{\epsilon N}  \max_{|z| \le M} \max_{0 \leq N' \leq N}  
\e_z\left[\ee^{\cH_{N'}(S)}\right].
\end{aligned}
\end{equation}
Combining \eqref{eq:xsplit} and \eqref{eq:Ixbds}, we obtain
\begin{equation}
Z^{\ann,+}_{N,\beta,h} \leq (1+\ee^{\epsilon N}) \max_{|z| \leq M} \max_{0 \leq N' \leq N}
\e_z\left[\ee^{\cH_{N'}(S)}\right].
\end{equation}
For every $|x|, |y| \leq M$ we have
\begin{equation}
\label{eq:sand}
\e_x\left[\ee^{\cH_{N'}(S)}\right] \geq \Pi_M\,\ee^{-M\|\psi_{\beta,h}\|_\infty}\, 
\e_y\left[\ee^{\cH_{N'}(S)}\right]
\end{equation}
with $\Pi_M = \prod_{x=-M}^{M-1} p_{x,x+1} \wedge \prod_{x=M}^{-M+1} p_{x,x-1}$. Since 
$\Pi_M> 0$, it follows that
\begin{equation}
\label{eq:Zannest1}
Z^{\ann,+}_{N,\beta,h} \leq \ee^{\epsilon N+o(N)} \max_{0 \leq N' \leq N} 
\e_0\left[\ee^{\cH_{N'}(S)}\right], \qquad N\to\infty.
\end{equation}
As shown in Appendix~\ref{app:compareann},  
\begin{equation}
\e_0\left[\ee^{\cH_{N'}(S)}\right] = Z^\ann_{N',\beta,h} 
= \ee^{o(N')} Z^{\ann,0,0}_{N',\beta,h}, \qquad N' \to\infty.
\end{equation}
Since $N \mapsto \log Z^{\ann,0,0}_{N,\beta,h}$ is super-additive, it follows that
\begin{equation}
\label{eq:Zannest2}
\max_{0 \leq N' \leq N} \e_0\left[\ee^{\cH_{N'}(S)}\right]
= \ee^{o(N)} \e_0\left[\ee^{\cH_N(S)}\right], \qquad N \to \infty.
\end{equation}
Combining \eqref{eq:Zannest1} and \eqref{eq:Zannest2}, we obtain
\begin{equation}
Z^{\ann,+}_{N,\beta,h} \leq \ee^{\epsilon N+o(N)} Z^{\ann,0,0}_{N,\beta,h},
\qquad N \to\infty.
\end{equation}
But 
\begin{equation}
Z^{\ann,0,0}_{N,\beta,h} \leq \e_0\left[\ee^{\cH_N(S)}\,1_{\{|S_N| \leq M\}}\right].
\end{equation}
Appealing once more to \eqref{eq:sand}, we arrive at
\begin{equation}
Z^{\ann,+}_{M,N,\beta,h} \leq \ee^{\epsilon N+o(N)} Z^{\ann,-}_{M,N,\beta,h},
\qquad N\to\infty,
\end{equation}
which proves the claim because $\epsilon>0$ is arbitrary.
\end{proof}

\begin{lemma}
\label{lem:Fannaltalt}
For $S,T \geq 0$, $\hat{\beta} \in (0,\infty)$ and $\hat{h} \in \R$, define
\begin{equation}
\begin{aligned}
\hat{Z}^{\ann,-}_{S,T,\hat{\beta},\hat{h}} &:= \min_{|x| \leq S}
\e_x\left[\ee^{\int_0^T dt\,\hat{\psi}_{\hat{\beta},\hat{h}}(X_t)} 1_{\{|X_T| \leq S\}}\right],\\
\hat{Z}^{\ann,+}_{T,\hat{\beta},\hat{h}} &:= \sup_{x\in\R} \max_{0 \leq T' \leq T} 
\e_x\left[\ee^{\int_0^{T'} dt\,\hat{\psi}_{\hat{\beta},\hat{h}}(X_t)}\right],
\end{aligned}
\end{equation}
with $\hat{\psi}_{\hat{\beta},\hat{h}}$ defined in \eqref{eq:psi0def}. Then, for all $\hat{\beta} \in (0,\infty)$ 
and $\hat{h} \in \R$,
\begin{equation}
\lim_{T\to\infty} \frac{1}{T} \log \hat{Z}^{\ann,-}_{S,T,\hat{\beta},\hat{h}}
= \hat{F}^\ann(\hat{\beta},\hat{h}) 
= \lim_{T\to\infty} \frac{1}{T} \log \hat{Z}^{\ann,+}_{T,\hat{\beta},\hat{h}} 
\qquad \forall\,S \in (0,\infty).
\end{equation}
\end{lemma} 

\begin{proof}
The proof is similar to that of Lemma~\ref{lem:Fannalt}.
\end{proof}

What makes Lemma~\ref{lem:Fannalt} useful is that $N \mapsto \log Z^{\ann,-}_{M,N,\beta,h}$
is superadditive for every $M \in \N_0$, while $N \mapsto \log Z^{\ann,+}_{N,\beta,h}$ is subadditive. 
Consequently,
\begin{equation}
\label{eq:hardsand}
\frac{1}{N} \log Z^{\ann,-}_{M,N,\beta,h} \leq F^\ann(\beta,h) 
\leq \frac{1}{N} \log Z^{\ann,+}_{N,\beta,h} \quad \forall\,M,N\in\N.
\end{equation}
Let $A,B$ be the pair of exponents appearing in part (a) of Theorems~\ref{thm:weak1}--\ref{thm:weak3}, 
i.e.,
\begin{equation}
(A,B) = \left\{\begin{array}{ll}
((1-\theta)/2,(2-\theta)/2), &\theta \in (0,1-\alpha),\\[0.1cm] 
(\alpha/2,(2-\theta)/2), &\theta \in (1-\alpha,2(1-\alpha)),\\[0.1cm] 
(\alpha/2,\alpha), &\theta \in (2(1-\alpha),\infty).
\end{array}
\right.
\end{equation}
Fix $T,S \in (0,\infty)$, in \eqref{eq:hardsand} replace $N$ by $TN$, and pick $M=S\sqrt{TN}$, 
$\beta=\hat{\beta}N^{-A}$, $h=\hat{h}N^{-B}$, to obtain
\begin{equation}
\begin{aligned}
&\frac{1}{T} \log Z^{\ann,-}_{S\sqrt{TN},TN,\hat{\beta}N^{-A},\hat{h}N^{-B}}
\leq N F^\ann(\hat{\beta}N^{-A},\hat{h}N^{-B})
\leq \frac{1}{T} \log Z^{\ann,+}_{TN,\hat{\beta}N^{-A},\hat{h}N^{-B}}\\ 
&\forall\,N\in\N,\,\,T,S \in (0,\infty).
\end{aligned}
\end{equation}
Let $N\to\infty$ to obtain
\begin{equation}
\label{eq:partlimN}
\begin{aligned}
&\frac{1}{T} \log \hat{Z}^{\ann,-}_{S,T,\hat{\beta},\hat{h}}
\leq \liminf_{N\to\infty} N F^\ann(\hat{\beta}N^{-A},\hat{h}N^{-B})\\
&\qquad \leq \limsup_{N\to\infty} N F^\ann(\hat{\beta}N^{-A},\hat{h}N^{-B})
\leq \frac{1}{T} \log \hat{Z}^{\ann,+}_{T,\hat{\beta},\hat{h}}
\qquad \forall\,T,S \in (0,\infty).
\end{aligned}
\end{equation}
Here, the fact that 
\begin{equation}
\begin{aligned}
\hat{Z}^{\ann,-}_{S,T,\hat{\beta},\hat{h}} 
&= \lim_{N\to\infty} Z^{\ann,-}_{S\sqrt{TN},TN,\hat{\beta}N^{-A},\hat{h}N^{-B}},\\
\hat{Z}^{\ann,+}_{T,\hat{\beta},\hat{h}}
&= \lim_{N\to\infty} Z^{\ann,+}_{TN,\hat{\beta}N^{-A},\hat{h}N^{-B}},
\end{aligned}
\end{equation}
follows from the same argument as used in Section~\ref{s:annealed} to prove that
\begin{equation}
\lim_{N\to\infty} Z^{\ann}_{TN,\hat{\beta}N^{-A},\hat{h}N^{-B}}
= \hat{Z}^{\ann}_{T,\hat{\beta},\hat{h}}.
\end{equation}
In particular, the scaling $M=S\sqrt{TN}$ fits well with the invariance principle in \eqref{eq:invprin}. 
Finally, let $T\to\infty$ in \eqref{eq:partlimN} and use Lemma~\ref{lem:Fannaltalt}, to obtain
\begin{equation}
\lim_{N\to\infty} N F^\ann(\hat{\beta}N^{-A},\hat{h}N^{-B}) 
= \hat{F}^\ann(\hat{\beta},\hat{h}). 
\end{equation}

%%%

\section{Properties of the annealed scaling limit}
\label{appB}

In this appendix we show that the annealed partition functions and the corresponding annealed free 
energies encountered in Theorems~\ref{thm:weak1}--\ref{thm:weak3} are finite in each of the three 
regimes. We also give an explicit characterization of the annealed free energy and the annealed 
critical curve in the regime where $\theta \in (2(1-\alpha),\infty)$. 

%%%

\subsection{Finite free energies for the Bessel process}
\label{app:finpartcont}

For $\mu,T,\gamma \in [0,\infty)$, define 
\begin{equation}
\hat{Z}_{\mu,T} := \hat{\e}_0\left[\exp\left(\mu \int_0^T \dd t\,X_t^{-\gamma}\right)\right] ,
\qquad \tilde{Z}_{\mu,T} := \hat{\e}_0\big[\exp\big(\mu \hat{L}_T(0)\big)\big] \,.
\end{equation}
We show that, for $0<\gamma < 2(1-\alpha)$, these quantities grow at most exponentially as 
$T \to \infty$:
\begin{equation}
\forall\,\mu \in [0,\infty), \, \forall\,0< \gamma < 2(1-\alpha): \quad
\limsup_{T\to\infty} \frac{1}{T} \, \log \, \hat{Z}_{\mu,T} < \infty, \quad
\limsup_{T\to\infty} \frac{1}{T} \, \log \, \tilde{Z}_{\mu,T} < \infty.
\end{equation}
By Cauchy-Schwarz, this implies that all the free energies in Theorems~\ref{thm:weak1}--\ref{thm:weak3} 
are finite.

We first focus on $\hat{Z}_{\mu,T}$, which we rewrite as
\begin{equation}
\hat{Z}_{\mu,T} = 1 + \sum_{k\in\N} \mu^k \hat{C}_{k,T}
\end{equation}
with
\begin{equation}
\hat{C}_{k,T} := \int_{0 \leq t_1<\cdots<t_k\leq T} \dd t_1 \cdots \dd t_k\,\,
\hat{\e}_0\left[\prod_{\ell=1}^k X_{t_\ell}^{-\gamma}\right].
\end{equation}
We use the Markov property at times $t_1,\ldots,t_k$ to estimate
\begin{equation}
\begin{aligned}
\hat{\e}_0\left[\prod_{\ell=1}^k X_{t_\ell}^{-\gamma}\right]
&= \int_{0 \leq x_1,\ldots,x_k < \infty}
\prod_{\ell=1}^k \hat{\p}_{x_{\ell-1}}(X_{t_\ell-t_{\ell-1}} \in \dd x_\ell)\,x_\ell^{-\gamma}\\
&\leq \int_{0 \leq x_1,\ldots,x_k < \infty}
\prod_{\ell=1}^k \hat{\p}_0(X_{t_\ell-t_{\ell-1}} \in \dd x_\ell)\,x_\ell^{-\gamma}\\
\end{aligned} 
\end{equation}
where the inequality holds because $\hat{\p}_0(X_t \in \cdot)$ stochastically dominates 
$\hat{\p}_x(X_t \in \cdot)$ for any $x \ge 0$ (by a standard coupling argument and the 
fact that $X$ is a Markov process with continuous paths) and because $x \mapsto x^{-\gamma}$ 
is non-increasing. We thus obtain
\begin{equation}
\hat{C}_{k,T} \leq \int_{0 \leq t_1<\cdots<t_k\leq T} \dd t_1 \cdots \dd t_k\,\,
\prod_{\ell=1}^k \hat{\e}_0\left[X_{t_\ell-t_{\ell-1}}^{-\gamma}\right] 
\end{equation}
with $t_0=0$. By diffusive scaling we have $\hat{\e}_0[X_t^{-\gamma}] = t^{-\gamma/2}C$, with 
$C=\hat{\e}_0[X_1^{-\gamma}]<\infty$ for $0<\gamma<2(1-\alpha)$ (recall \eqref{eq:gt}). 
The change of variables $t_k = T s_k$ yields
\begin{equation}
\label{eq:hatZsumbd}
\hat{Z}_{\mu,T} \leq 1 + \sum_{k\in\N} (\mu C \, T^{1-\gamma/2})^k\, I_{k}(\gamma/2)
\end{equation}
where, for $\theta \in (0, 1)$,
\begin{equation}
\label{eq:Ikdef}
\begin{split}
I_{k}(\theta) &:= \int_{0 \leq s_1 < \cdots < s_k \leq 1} \dd s_1 \cdots \dd s_k\,\,
\prod_{\ell=1}^k (s_\ell - s_{\ell-1})^{-\theta}
\end{split}
\end{equation}
with $s_0=0$. Since 
\begin{equation}\label{eq:Dirichlet}
\int_{0 \le u_1 < \ldots < u_{k-1} \le 1} \dd u_1\cdots \dd u_k\,\, 
\prod_{\ell=1}^k (u_\ell - u_{\ell-1})^{-\theta}
= \frac{\Gamma(1-\theta)^k}{\Gamma(k(1-\theta))}
\end{equation}
with $u_0=1$ and $u_k=1$ is the normalization of the Dirichlet distribution, after setting $s_i = s_k \, u_i$ 
for $i=1,\ldots, k-1$ in \eqref{eq:Ikdef} we obtain
\begin{equation}
\label{eq:Ikasymp}
\begin{split}
I_{k} (\theta)
& = \int_{0 \le s_k \le 1} \dd s_k \, s_k^{k(1-\theta)-1} \,
\frac{\Gamma(1-\theta)^k}{\Gamma(k(1-\theta))}
= \frac{\Gamma(1-\theta)^k}{\Gamma(k(1-\theta)+1)} \\
& \le \Gamma(1-\theta)^k \, \exp\Big(- (1-\theta)k \, \big\{\log [(1-\theta)k] - 1\big\}\Big),
\end{split}
\end{equation}
where we have used Stirling's 
bound $\Gamma(x+1) \ge \ee^{x (\log x-1)}$. Substitute $\theta = \gamma/2$ 
and set
\begin{equation}
A := (1-\gamma/2), \quad B := \log \big( \mu \, C \, \Gamma(1-\gamma/2) \big),
\end{equation}
to obtain
\begin{equation} 
\label{eq:finalsum}
\begin{split}
\hat{Z}_{\mu,T} & \leq 1 + \sum_{k\in\N} 
\exp \Big( - Ak \big[ \log \big( A k \big) - 1 \big] + Ak \log T + Bk \Big), \\
& = 1 + \sum_{k\in\N} 
\exp \Big( T \cdot \tfrac{k}{T} \Big\{- A \big[ \log \big( A \tfrac{k}{T} \big) - 1 \big] + B \Big\} \Big).
\end{split}
\end{equation}
We now set $x := k/T$ and note by direct computation that, with $\bar x := A^{-1} \ee^{B/A}$,
\begin{equation}
\sup_{x \in [0,\infty)} x \{-A \big[ \log (A x) - 1 \big] + B\}
= \bar x \{-A \big[ \log (A \bar x) - 1 \big] + B\} = \ee^{B/A}.
\end{equation}
Therefore the leading contribution to the sum in \eqref{eq:finalsum} is given by $k \approx \bar x T$
(more precisely, the sum restricted to $0 \le k < 2 \bar x T$ is at most $ 2 \bar x \, T \, \exp(T \, \ee^{B/A})$, 
while the contribution of the remaining terms with $k \ge 2\bar x T$ is negligible). It follows that
\begin{equation}
\label{eq:finalhatZ}
\limsup_{T\to\infty} \frac{1}{T} \, \log \hat{Z}_{\mu,T}
\le \ee^{B/A} = \big[ \mu \, C \, \Gamma(1-\gamma/2) \big]^{\frac{1}{1-\gamma/2}} < \infty.
\end{equation}

We next focus on $\tilde{Z}_{\mu,T}$, which equals  
\begin{equation}
\tilde{Z}_{\mu,T} = 1 + \sum_{k\in\N} \mu^k \bar{C}_{k,T}
\end{equation}
with
\begin{equation}
\label{eq:CkT}
\bar{C}_{k,T} := \frac{1}{k!} \hat{\e}\left[ \hat{L}_T(0)^k \right].
\end{equation}
We recall from \eqref{eq:hatLdef} that we can write $\hat{L}_T(0) = \lim_{\epsilon \downarrow 0} 
\hat{L}_T^\epsilon(0)$ in probability, where
\begin{equation}
\hat{L}_T^\epsilon(0) := \frac{c_\alpha}{\epsilon^{2(1-\alpha)}}
\int_0^T \dd s\,1_{\{X_s \in (0,\epsilon)\}} , \qquad
\text{with} \qquad c_\alpha := \frac{\Gamma(2-\alpha)}{2^{\alpha-1}}.
\end{equation}
Let us focus on $\hat{L}_T^\epsilon(0)$ for a moment. By explicit computation, for any $k\in\N$
(with $t_0 := 0$)
\begin{equation} \label{eq:ineqin}
\begin{split}
\frac{1}{k!} \, \hat{\e} \left[ \hat{L}_T^\epsilon(0)^k \right]
& = \frac{c_\alpha^k}{\epsilon^{2(1-\alpha)k} } \,
\int_{0 \leq t_1<\cdots<t_k\leq T} \dd t_1 \cdots \dd t_k\,\,
\hat{\p}_0\bigg( \bigcap_{\ell=1}^k \{ X_{t_\ell} \in (0,\epsilon) \} \bigg) \\
& \le \frac{c_\alpha^k}{\epsilon^{2(1-\alpha)k} } \,
\int_{0 \leq t_1<\cdots<t_k\leq T} \dd t_1 \cdots \dd t_k\,\,
\prod_{\ell=1}^k \hat{\p}_0\big( X_{t_\ell - t_{\ell-1}} \in (0,\epsilon) \big) \\
&\le \int_{0 \leq t_1<\cdots<t_k\leq T} \dd t_1 \cdots \dd t_k\,\,
\prod_{\ell=1}^k (t_\ell - t_{\ell-1})^{-(1-\alpha)} 
=  T^\alpha \, I_k(1-\alpha),
\end{split}
\end{equation}
where the first inequality holds because $\hat{\p}_0(X_t \in \cdot)$ stochastically dominates 
$\hat{\p}_x(X_t \in \cdot)$ for any $x \ge 0$, while the last inequality holds by \eqref{eq:PXepsilon} 
and the definition \eqref{eq:Ikdef} of $I_k(\cdot)$. This shows that, as $\epsilon \downarrow 0$,
\emph{$\hat{L}_T^\epsilon(0)$ is uniformly bounded in $L^k$, for any $k\in\N$}. By uniform integrability, 
we can therefore exchange $\lim_{\epsilon \downarrow 0}$ and $\hat{\e}$ to get, recalling \eqref{eq:CkT},
\begin{equation} 
\label{eq:CkT2}
\bar{C}_{k,T} = \lim_{\epsilon\downarrow 0} \, \frac{1}{k!} 
\hat{\e}\left[ \hat{L}_T^\epsilon(0)^k \right]
= T^\alpha \, I_k(1-\alpha),
\end{equation}
where the last equality holds because the inequality in \eqref{eq:ineqin} becomes sharp as $\epsilon \downarrow 0$.
Hence
\begin{equation}
\label{eq:ZmuT}
\begin{split}
\tilde{Z}_{\mu,T}
&= 1 + \sum_{k\in\N} \mu^k \int_{0 \leq t_1<\cdots<t_k\leq T} \dd t_1 \cdots \dd t_k\,\,
\prod_{\ell=1}^k (t_\ell - t_{\ell-1})^{-(1-\alpha)} \\
& = 1 + \sum_{k\in\N} (\mu T^\alpha)^k I_{k}(1-\alpha).
\end{split}
\end{equation}
The steps in \eqref{eq:Ikasymp}-\eqref{eq:finalhatZ} (with $\theta = 1-\alpha$ instead of $\theta=\gamma/2$) 
show that not only $\tilde{Z}_{\mu,T}<\infty$ for all $\mu,T$, but also $\lim_{T\to\infty} \frac{1}{T} \log 
\tilde{Z}_{\mu,T} < \infty$ for all $\mu$.

%%%

\subsection{Formula for the annealed free energy}
\label{app:annfeexplicit}

To compute the annealed free energy $\hat{F}(\hat{\beta},\hat{h})$ in the regime $\theta \in (2(1-\alpha),\infty)$, 
we use the first line in \eqref{eq:ZmuT} to compute the Laplace transform of  $\tilde{Z}_{\mu,T}$. Writing 
$\ee^{-\lambda T} = (\prod_{\ell=1}^k \ee^{-\lambda (t_\ell - t_{\ell-1})})\, \ee^{-\lambda(T-t_k)}$ for $\lambda \ge 0$, 
we get
\begin{equation}
\begin{split}
\int_0^\infty \dd T\, \lambda \, \ee^{-\lambda T} \tilde{Z}_{\mu,T}
& = 1 + \sum_{k\in\N} \bigg( \mu \int_0^\infty t^{\alpha-1} \, e^{-\lambda t} \, \dd t \bigg)^k 
= \frac{1}{1- \mu \, \lambda^{-\alpha} \, \Gamma(\alpha)}.
\end{split}
\end{equation}
Hence
\begin{equation}
F(\mu) = \lim_{T\to\infty} \frac{1}{T} \log \tilde{Z}_{\mu,T}  
= \begin{cases} 
(\mu \Gamma(\alpha))^{1/\alpha}, &\mu \ge 0, \\
0, &\mu < 0,
\end{cases}
\end{equation}
which proves \eqref{eq:Fannid}.

%%%

\section{Localization criterion for the annealed model}
\label{appC}

In this appendix we take a closer look at the criterion in \eqref{eq:condloc} and show that it does not 
depend on the starting point of the walk. For the purpose of this appendix, let $S = (S_n)_{n\in\N_0}$  
be any recurrent Markov chain on a countable space $E$, and let $\psi\colon\, E \to \R$ be an arbitrary 
function. Denote by $\tau^x := \min\{n\in\N\colon\, S_n = x\}$ the first return time of $S$ to $x$. Define
\begin{equation} 
\label{eq:condlocapp}
A_x := \sum_{m\in\N} \e_x \bigg[ \ee^{\sum_{n=1}^{m} \psi(S_n)} \, 
\ind_{\{\tau^x = m\}} \bigg] \in (0, \infty].
\end{equation}
We will prove the following property:
\begin{equation} 
\label{eq:AxAy}
\forall x, y\colon \qquad A_x > 1 \quad \iff \quad A_y > 1.
\end{equation}

It is convenient to introduce the following shorthand notation, for (possibly random) $\sigma \in \N$:
\begin{equation}
\cH(\sigma) := \sum_{n=1}^\sigma \psi(S_n),
\end{equation}
so that we may simply write $A_x = \e_x [\ee^{\cH(\tau^x)}]$. Given an arbitrary $y$, we can split this 
expected value according to the two complementary events $\{\tau^x < \tau^y\}$ and $\{\tau^y < \tau^x\}$: 
\begin{equation}
A_x = \e_x \big[ \ee^{\cH(\tau^x)} \, \ind_{\{\tau^x < \tau^y\}} \big]
+ \e_x \big[ \ee^{\cH(\tau^x)} \, \ind_{\{\tau^y < \tau^x\}} \big].
\end{equation}
The second term can be expanded by summing over all visits of $S$ to $y$ that precede the first return to 
$x$. If we define
\begin{equation}
B_{xy} := \e_x \big[ \ee^{\cH(\tau^x)} \, \ind_{\{\tau^x < \tau^y\}} \big], \qquad
C_{xy} := \e_x \big[ \ee^{\cH(\tau^y)} \, \ind_{\{\tau^y < \tau^x\}} \big],
\end{equation}
then by the strong Markov property we get
\begin{equation}
A_x = B_{xy} + \sum_{\ell=0}^\infty C_{xy} \, (B_{yx})^\ell \, C_{yx}
= B_{xy} + \frac{C_{xy} \, C_{yx}}{1-B_{yx}},
\end{equation}
\emph{with the convention that $A_x = \infty$ if $B_{yx} \ge 1$}. Exchanging the roles of $x$ and $y$,
we get
\begin{equation}
A_y = B_{yx} + \frac{C_{xy} \, C_{yx}}{1-B_{xy}}.
\end{equation}

We are now ready to prove \eqref{eq:AxAy}. Fix $x,y$. We show that if $A_x \le 1$, then also $A_y 
\le 1$. To simplify the notation, we abbreviate $b := B_{xy}$, $b' := B_{yx}$ and $c := C_{xy} C_{yx}$, so 
that
\begin{equation}
A_x = b + \frac{c}{1-b'}, \qquad A_y = b' + \frac{c}{1-b},
\end{equation}
with the convention that the ratios equal $\infty$ if $b' \ge 1$, respectively, $b \ge 1$. Assume that 
$A_x \le 1$. Then we must have $b' < 1$ and the formula $A_x = b + \frac{c}{1-b'}$ applies, which shows 
that also $b < 1$ (because $c > 0$). Hence we can write
\begin{equation}
A_y = b' + \frac{c}{1-b} = b' + \frac{1-b'}{1-b} \frac{c}{1-b'}
= b' + \frac{1-b'}{1-b}(A_x-b) \le b' + \frac{1-b'}{1-b} (1-b) = 1,
\end{equation}
i.e., $A_y \le 1$.

%%%%%%%%%%% The bibliography %%%%%%%%%%%%

\end{document}